\documentclass[11pt]{amsart}
\usepackage{float}
\usepackage{amsfonts, amstext, amsmath, amsthm, amscd, amssymb, upgreek}
\usepackage{bbm}
\usepackage{graphicx, color,  subfigure, wrapfig, overpic}
\usepackage{fullpage}

\usepackage{tikz}
\usetikzlibrary{shapes.geometric, arrows}
\usetikzlibrary{shapes.multipart} 
\usetikzlibrary{calc}
\usetikzlibrary{scopes}
\usetikzlibrary{math}
\usetikzlibrary{decorations.markings,decorations.pathreplacing}
\usetikzlibrary{fadings}
\usepackage{tikz-cd}

\calclayout

\tikzset{
every picture/.style={line width=0.8pt, >=stealth,
                       baseline=-3pt,label distance=-3pt},
emptynode/.style={circle,minimum size=0pt, inner sep=0pt, outer
sep=0},
dotnode/.style={fill=black,circle,minimum size=2.5pt, inner sep=1pt, outer
sep=0},
small_dotnode/.style={fill=black,circle,minimum size=2pt, inner sep=0pt, outer
sep=0},
morphism/.style={fill=white,circle,draw,thin, inner sep=1pt, minimum size=15pt,
                 scale=0.8},
small_morphism/.style={fill=white,circle,draw,thin,inner sep=1pt,
                       minimum size=10pt, scale=0.8},
ellipse_morphism/.style args={#1}{fill=white,circle,draw,thin,inner sep=1pt,
                       minimum size=5pt, scale=0.8,
												ellipse, draw, rotate=#1},
coupon/.style={draw,thin, inner sep=1pt, minimum size=18pt,scale=0.8},
semi_morphism/.style args={#1,#2}{
                  fill=white,semicircle,draw,thin, inner sep=1pt, scale=0.8,
                  shape border rotate=#1,
                  label={#1-90:#2}},
regular/.style={densely dashed}, 
edge/.style={very thick, draw=green, text=black},
overline/.style={preaction={draw,line width=2mm,white,-}},
thin_overline/.style={preaction={draw,line width=#1 mm,white,-}},
thin_overline/.default=2,
thick_overline/.style={preaction={draw,line width=3mm,white,-}},
really_thick/.style={line width=3mm, gray},
boundary/.style={thick,  draw=blue, text=black},
ribbon/.style={line width=1.5mm, postaction={draw,line width=1mm,white}},
ribbon_u/.style args={#1,#2}{line width=#1mm, postaction={draw,line width=#2mm,white}},
cell/.style={fill=black!10},
subgraph/.style={fill=black!30},
midarrow/.style={postaction={decorate},
                 decoration={
                    markings,
                    mark=at position #1 with {\arrow{>}},
                 }},
midarrow/.default=0.5,
midarrow_rev/.style={postaction={decorate},
                 decoration={
                    markings,
                    mark=at position #1 with {\arrow{<}},
                 }},
midarrow_rev/.default=0.5,
midarrow_thin/.style={postaction={decorate},
                 decoration={
                    markings,
                    mark=at position #1 with {\arrow{Classical TikZ
											Rightarrow[length=2pt]}},
                 }},
midarrow_thin/.default=0.5,
block/.style={rectangle, rounded corners, text centered, draw=black, align=center}
}

\tikzstyle{block} = [rectangle, rounded corners, text centered, draw=black, align=center]

\tikzfading[name=fade inside, inner color=transparent!80, outer
color=transparent!10]

\definecolor{light-gray}{gray}{0.9}
\definecolor{med-gray}{gray}{0.6}

\newcommand{\thmref}[1]{Theorem \ref{#1}}
\newcommand{\prpref}[1]{Proposition \ref{#1}}
\newcommand{\lemref}[1]{Lemma \ref{#1}}
\newcommand{\defref}[1]{Definition \ref{#1}}
\newcommand{\figref}[1]{Figure \ref{#1}}
\newcommand{\secref}[1]{Section \ref{#1}}

\newcommand{\comment}[1]{}

\let\oldmarginpar\marginpar
\renewcommand\marginpar[1]{\oldmarginpar[\raggedleft\footnotesize #1]%
{\raggedright\footnotesize #1}}

\AtBeginDocument{
   \def\MR#1{}
}
\newcommand{\Sp}{{S}}

\newcommand{\R}{\mathbb{R}}

\newcommand{\torus}{{\mathbb{T}^2}}
\newcommand{\sT}{{\mathcal{T}}}
\newcommand{\cD}{{\mathcal{D}}}
\newcommand{\cL}{{\mathcal{L}}}

\newcommand{\cev}[1]{\overset{\leftarrow}{#1}}

\newcommand{\del}{\partial}

\newcommand{\vphi}{\varphi}
\newcommand{\veps}{\varepsilon}

\newcommand{\toruscomp}[1]{{\torus \times I - #1}}

\theoremstyle{plain}
\newtheorem{theorem}{Theorem}[section]

\newtheorem{lemma}[theorem]{Lemma}
\newtheorem{prop}[theorem]{Proposition}

\newtheorem{convention}[theorem]{Convention}

\newtheorem*{namedtheorem}{\theoremname}
\newcommand{\theoremname}{testing}

\theoremstyle{definition}
\newtheorem{define}[theorem]{Definition}
\newtheorem{definition}[theorem]{Definition}

\newtheorem{remark}[theorem]{Remark}

\newcommand{\cm}{,\linebreak[1]}

\title{Hyperbolicity of Augmented Links in the Thickened Torus}

\author[Alice Kwon and Ying Hong Tham]{Alice Kwon and Ying Hong Tham}

\begin{document}
\maketitle


\begin{abstract}
For a hyperbolic link $K$ in the thickened torus,
we show there is a decomposition of the complement of a link $L$,
obtained from augmenting $K$, into torihedra. We further decompose 
the torihedra into angled pyramids and finally angled tetrahedra.
These fit into an angled structure on a triangulation of the link complement,
and thus by \cite{Casson-Rivin}, this shows
that $L$ is hyperbolic.  
\end{abstract}

\section{Introduction}
\label{s:intro}

Given a twist-reduced diagram of a link $K$, \emph{augmenting} is a process in
which an unknotted circle component (augmentation) is added to one or more twist
regions (a single crossing or a maximal string of bigons) of $K$.
The newly obtained link is called an 
\emph{augmented link} and the newly obtained diagram is called an 
\emph{augmented link diagram}. See Figure
\ref{fig:Augmentations}.

Adams showed in \cite{CA} that given a hyperbolic alternating link $K$ in
$\Sp^3$ the link $L$ obtained by augmenting $K$ is hyperbolic. In this paper we
investigate if this statement holds for links in the thickened torus i.e. if $L$
is a link obtained from augmenting a hyperbolic alternating link $K$ in the
thickened torus. We define augmenting similarly for links in the thickened torus
with their associated link diagram on $\torus \times \{0\}.$

Menasco \cite{Menasco} showed that there are decompositions
of the complements of alternating links in $\Sp^3$
into two topological polyhedra,
a top polyhedron and a bottom polyhedron.
For alternating links $K$ in the thickened torus,
Champanerkar, Kofman and Purcell \cite{CKP2}
showed that there is a decomposition of the
complement of $K$ into objects called torihedra, which we think of as
counterparts to Menasco's decomposition 
for links in the thickened torus;
just like Menasco's decomposition, one obtains a top and a bottom torihedron.

In \secref{s:auglinks} we show that for augmented
links in the thickened torus (not necessarily fully augmented),
one can also obtain a decomposition of the
complement into a top and bottom torihedron.
In \secref{s:hyperbolicity}, we prove that
many augmented alternating links in the thickened torus are hyperbolic.

We point out that in \cite{kwon2020fully}, the first author
proved that \emph{fully} augmented links in the thickened torus
are hyperbolic, so this paper can be seen as a generalization
of that work.

While revising this paper, we learned that \cite{adams2021generalized}
proves a generalization of our work here,
showing hyperbolicity of generalized augmented links
in an arbitrary thickened surface.
We note that our approach, based on angle structures,
is different from theirs, which is based on topological arguments.

\begin{figure}
\centering  
\begin{tabular}{cc}
\includegraphics[height=4cm]{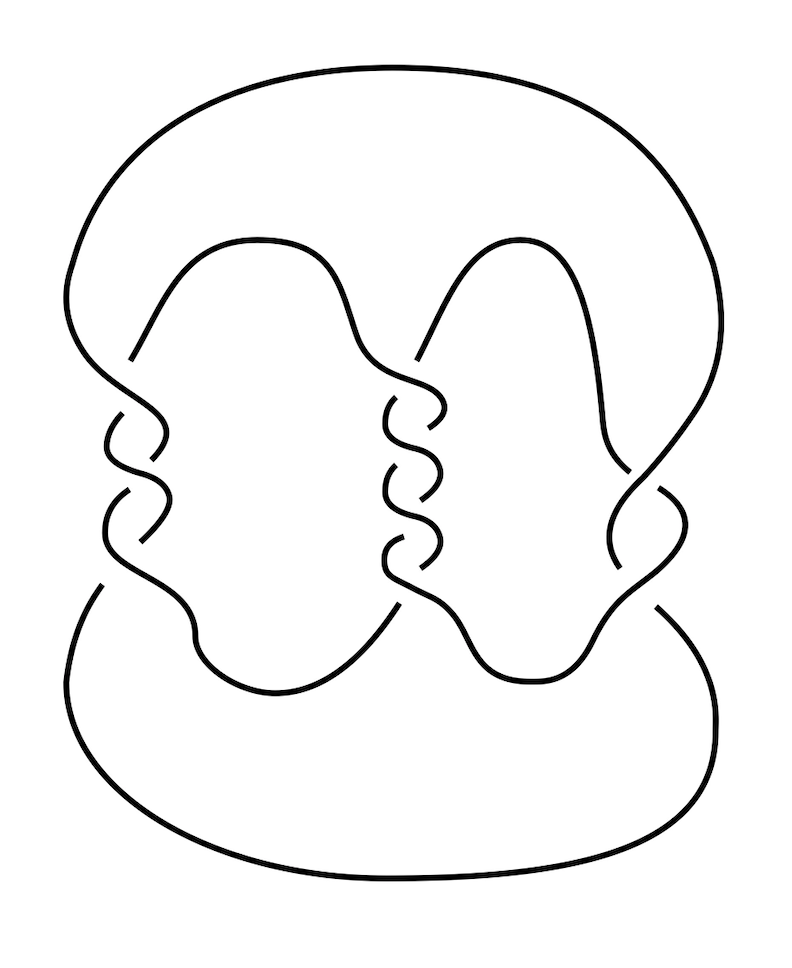}
\;\;\;\;
& \includegraphics[height=4cm]{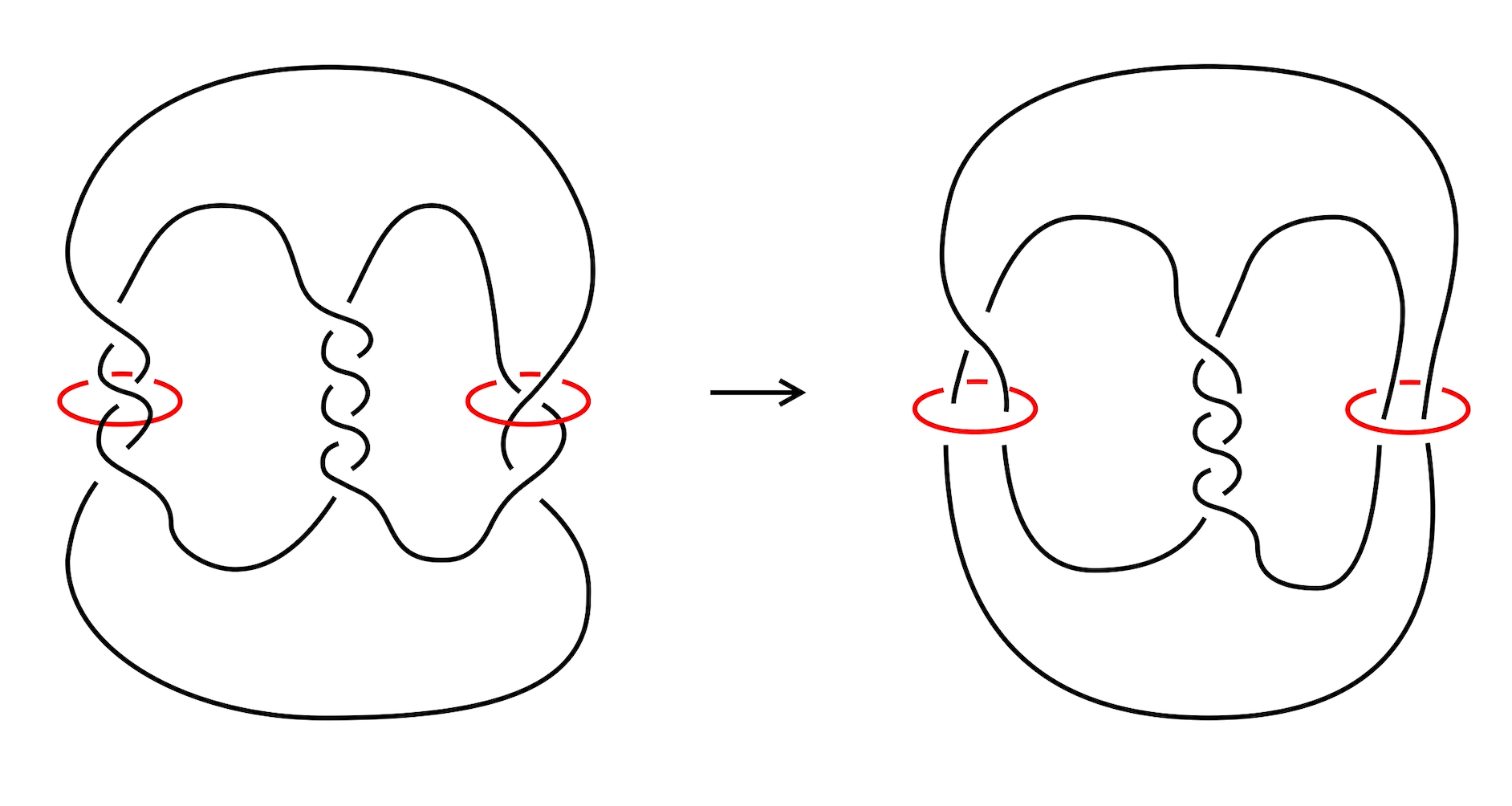}
\\
(a)
& (b)
\end{tabular}
\caption{(a) pretzel knot before augmentation
(b) pretzel knot after augmentation;
second diagram shows
removal of full twists in the augmented twist regions.}
\label{fig:augmentationS3}
\end{figure}

\section{Augmented Links}
\label{s:auglinks}

We denote $I = (-1,1)$.

Champanerkar, Kofman and Purcell have studied alternating links in the thickened
torus \cite{CKP2}. They define a link in the thickened torus as a quotient of a
biperiodic alternating link as follows:
 
\begin{define}
\label{def:biperiodiclink}
A \emph{biperiodic alternating link} $\cL$ is an infinite link
in $\R^2 \times I$ with a link diagram $\cD \subset \R^2$
such that $\cL$ and $\cD$ are
invariant under the action of a two dimensional lattice $\Lambda$
on $\R^2$ by translations.

The quotient $L=\mathcal{L}/\Lambda$ is an alternating link in
the thickened torus $\torus \times I$,
whose projection onto $\torus \times \{0\} = \R^2 \times \{0\} /\Lambda$
is an alternating link diagram $\cD/\Lambda$.
\end{define}

We refer to $\torus \times \{0\}$ as the \emph{projection plane}.

\begin{remark}
Since $\torus \times I \cong \Sp^3 - H$, where $H$ is a Hopf link.
The complement $\torus \times I- L = \Sp^3 - (L \cup H)$.
\end{remark}

Champanerkar, Kofman and Purcell \cite{CKP2} extended
the definition of prime links in $\Sp^3$
for links in $\torus \times I$ called weakly prime. 

\begin{define} \label{def:weaklyprime}
A diagram $D \subset \torus$
of a link $L$ in the thickened torus $\torus \times I$
is \emph{weakly prime}
if whenever a disk is embedded in $\torus$
meets the diagram transversely in exactly two edges,
then the disk contains a simple edge of the diagram and no crossings.
\end{define}

\begin{define}
Recall that a \emph{twist region}
in a link diagram in the plane is a maximal sequence of vertices
such that consecutive vertices are two vertices of a bigon face,
and consecutive bigons meet at exactly one vertex.
The \emph{length} of the twist region is the number of bigons.

We say a single vertex $v$ is a \emph{trivial twist region}
(or \emph{twist region of length 0})
if, for some cyclical ordering $f_1,f_2,f_3,f_4$
of the four faces adjacent to $v$,
$f_1,f_3$ are not bigons.
Note that if all four faces are not bigons,
then we think of $v$ as being a trivial twist region in two ways,
and in this sense, every vertex is part of exactly
two twist regions.
%

For links in the thickened torus,
a \emph{twist region} in a link diagram of $L=\mathcal{L}/\Lambda$ in $\torus
\times I$, is the quotient of a twist region in the biperiodic link
$\mathcal{L}$. 
\end{define}

\begin{define}
A biperiodic link $\mathcal{L}$ is called \emph{twist-reduced}
if for any simple closed curve on
the plane that intersects the diagram of $\mathcal{L}$
transversely in four points,
with two points adjacent to one crossing
and the other two points adjacent to another crossing,
the simple closed curve bounds a subdiagram consisting of
a collection of bigons strung end to end between these crossings.
We say the diagram of $L$ is \emph{twist-reduced}
if it is the quotient of a twist-reduced biperiodic link diagram.
\end{define}

Now we can define augmentation for a link in $\torus \times I$ the same way we
define augmentation for links in $\Sp^3$:

\begin{definition}
Let $D(K)$ be a twist-reduced diagram of a link $K$ in $\torus \times I$.
We define
\emph{augmenting} as a process in which an unknotted circle component,
called a \emph{crossing circle},
is added to one or more twist regions of $D(K)$
(see Figure \ref{fig:Augmentations});
we call the resulting link $L$ an \emph{augmented link obtained from $K$}.
We say $L$ is \emph{fully augmented} if $L$ is obtained by augmenting
$K$ at \emph{every} crossing/twist region.
\label{d:augmentation}
\end{definition}

As pointed out in the introduction,
after augmenting a twist region,
a standard Dehn twist argument allows us to remove
a full twist (that is, two bigons).

\begin{definition}
We say an augmentation has a \emph{half twist}
if at least one of the augmented twist regions
has an odd number of vertices (i.e. even number of bigons).
\label{d:half-twist}
\end{definition}

\begin{definition}
A graph $G = (V,E)$ on the torus is \emph{cellular}
if its complement is a collection of open disks.
\label{d:cellular}
\end{definition}

We note that when a link diagram is cellular,
a twist region in the torus cannot be a cycle;
otherwise, the face adjacent to the twist region would
have non-trivial homology.



\begin{figure}
\centering
\begin{tabular}{cccc}
\includegraphics[width=3cm]{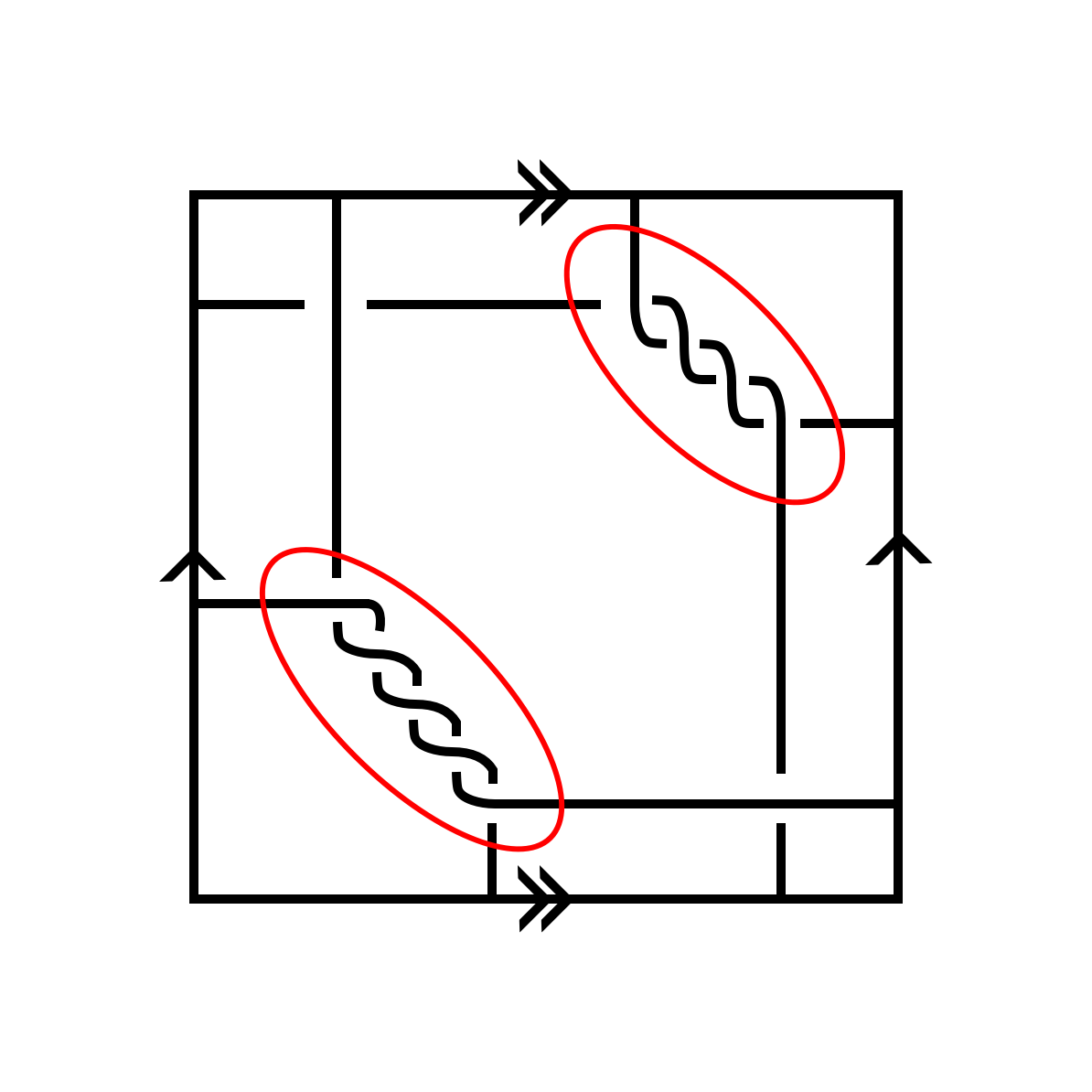}&
\includegraphics[width=3cm]{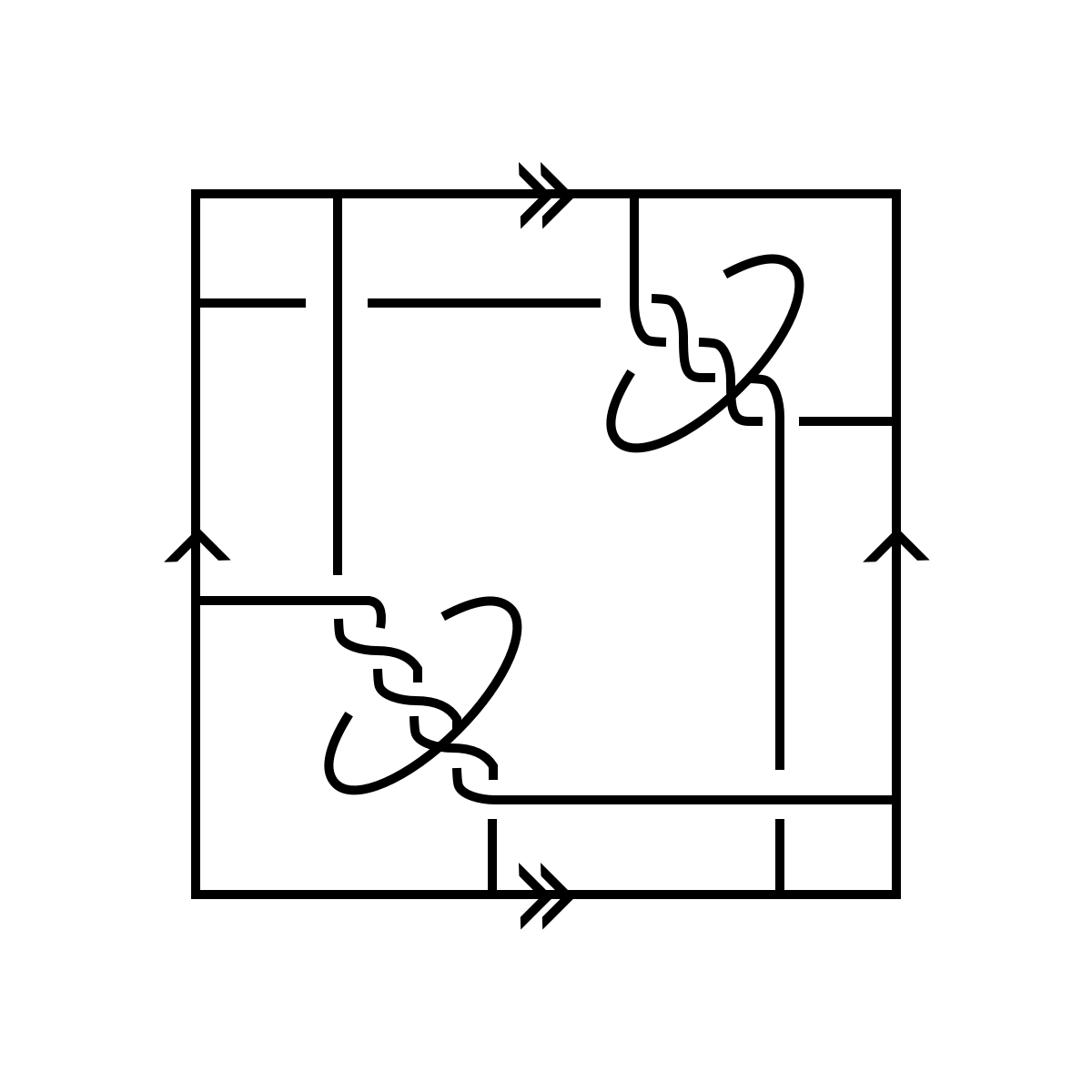}&
\includegraphics[width=3cm]{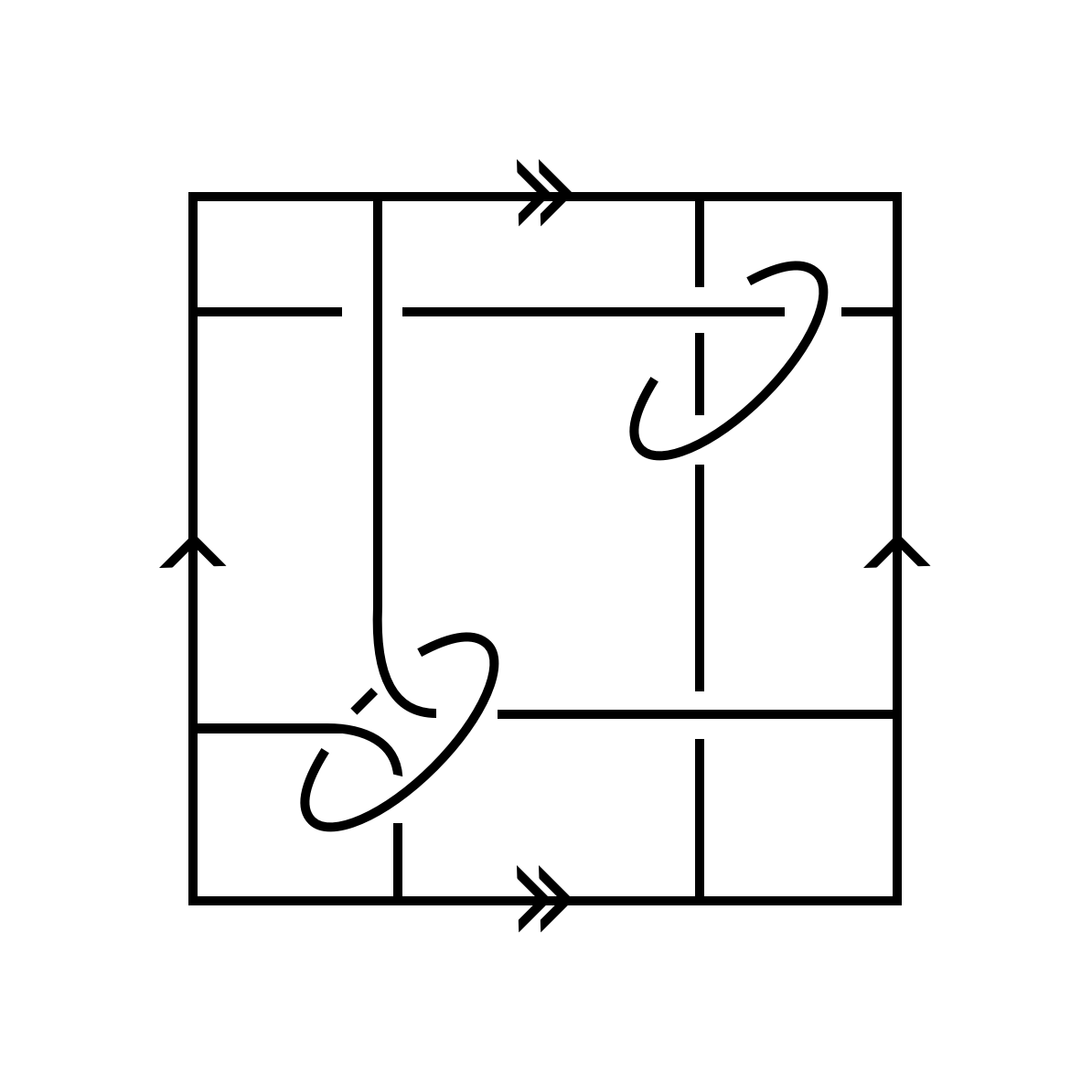}\\
A&B&C
\end{tabular}
	 \caption{A: The top right has an odd number of twists while the bottom left has
	 an even number of twists. B: The picture of the link on the right after
	 augmentation twist regions circled in red. C: The link with full twists
	 removed.}
\label{fig:Augmentations}
\end{figure}

\subsection{Torihedral Decomposition of Augmented Alternating Links in Thickened Torus}

We present a method of decomposing an augmented link
(not necessarily fully augmented) in the thickened torus into
objects called ``torihedra'' as defined below.
Decomposing alternating links in the thickened torus into 
torihedra were first described in \cite{CKP2},
then later used for fully augmented links in 
the thickened torus in \cite{kwon2020fully}.
The idea is to combine methods of Menasco \cite{Menasco}
and the use of crossing edges at each crossing of our link
and Lackenby's ``cut-slice-flatten" method \cite{lackenby}
on the augmentation sites.

\begin{define}\cite{CKP2}
\label{def:torihedron}
A \emph{torihedron} $\sT$ is a cone on the torus, 
i.e. $\torus \times [0,1]/(\torus \times \{1\})$, with a cellular graph
$G = G(\sT)$ on $\torus \times \{0\}$.
The \emph{ideal torihedron} $\sT^\circ$ is $\sT$ with the
vertices of $G$ and the vertex $\torus \times \{1\}$ removed. Hence, an ideal
torihedron is homeomorphic to $\torus \times [0,1)$ with a finite set of points
(ideal vertices) removed from $\torus \times \{0\}$.
We refer to the vertex $\torus \times \{1\}$ as the \emph{cone point}
of $\sT$.
\end{define}

For visualization purposes, we typically draw the graph $G(\sT)$ of a
torihedron from the perspective of the cone point $\torus \times \{1\}$.
Note however that later we will be dealing with
``top'' and ``bottom'' torihedra that are glued together
along their torus boundary faces;
to avoid confusion, we will visualize the graphs of both torihedra
from the perspective of the cone point of the ``top'' torihedron.

Since the faces of $G(\sT)$ are disks,
$\sT$ can be decomposed into a union of pyramids,
where each pyramid is obtained by coning a face of $G(\sT)$
to the cone point of $\sT$.
This also gives a decomposition of the corresponding ideal torihedron
$\sT^\circ$ into ideal pyramids.
We call these the \emph{pyramidal decompositions} of $\sT$ and $\sT^\circ$.

\begin{definition}
Let $G$ be a graph on the torus.
Let $v$ be a vertex and $e,e'$ be distinct edges that meet $v$.
A \emph{bow-tie modification to $v,e,e'$}
is the process of removing $v,e,e'$ and adding in
a pair of triangular faces, which we refer to as
\emph{bow-tie faces} or a \emph{bow-tie}
(see Figure \ref{f:bow-tie-modification},
\ref{f:bow-tie-modification-general}).
The edges of the bow-tie are of three types;
\emph{diagonal} edges, which do not touch $v$,
\emph{long} edges, which are ``parallel'' to the original edges $e,e'$,
and \emph{short} edges.
A bow-tie modification is a \emph{left} bow-tie modification
if, traveling from $v$ out along $e$ and $e'$,
the diagonal faces appear to the left of $e$ and $e'$,
respectively;
likewise for \emph{right} bow-tie modifications.

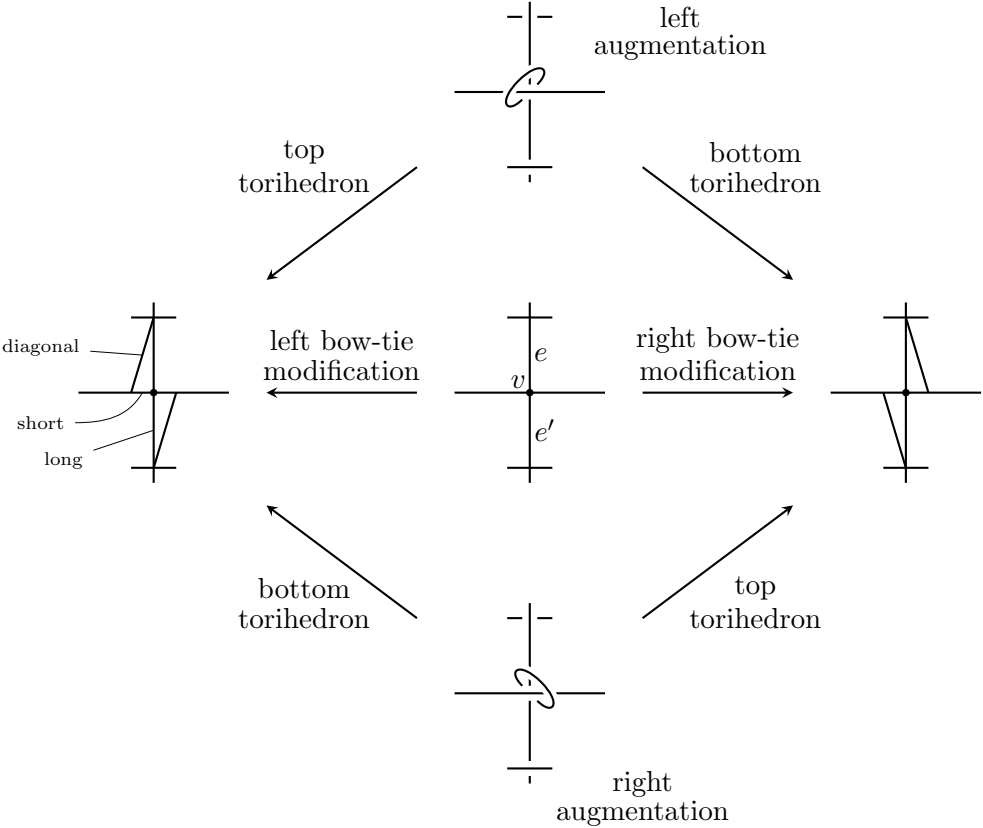
\begin{figure}[ht]
\begin{tikzpicture}
\begin{scope}[shift={(0,0)}]
\node[dotnode] at (0,0) {};
\draw (-1,0) -- (1,0);
\draw (0,-1.2) -- (0,1.2);
\draw (-0.3,1) -- (0.3,1);
\draw (-0.3,-1) -- (0.3,-1);
\node at (-0.15,0.15) {$v$};
\node at (0.15,0.5) {$e$};
\node at (0.2,-0.5) {$e'$};
\draw[->] (-1.5,0) -- (-3.5,0);
\node at (-2.5,0.7) {left bow-tie};
\node at (-2.5,0.3) {modification};
\draw[->] (1.5,0) -- (3.5,0);
\node at (2.5,0.7) {right bow-tie};
\node at (2.5,0.3) {modification};
\end{scope}
\begin{scope}[shift={(-5,0)}]
\node[dotnode] at (0,0) {};
\draw (-1,0) -- (1,0);
\draw (0,-1.2) -- (0,1.2);
\draw (-0.3,1) -- (0.3,1);
\draw (-0.3,-1) -- (0.3,-1);
\draw (-0.3,0) -- (0,1);
\draw (0.3,0) -- (0,-1);
\node (diag) at (-1.5,0.6) {\tiny diagonal};
\node (short) at (-1.5,-0.4) {\tiny short};
\node (long) at (-1.2,-0.9) {\tiny long};
\draw[line width=0.3pt] (diag) -- (-0.15,0.5);
\draw[line width=0.3pt] (short) to[out=0,in=-120] (-0.15,0);
\draw[line width=0.3pt] (long) -- (0,-0.5);
\end{scope}
\begin{scope}[shift={(5,0)}]
\node[dotnode] at (0,0) {};
\draw (-1,0) -- (1,0);
\draw (0,-1.2) -- (0,1.2);
\draw (-0.3,1) -- (0.3,1);
\draw (-0.3,-1) -- (0.3,-1);
\draw (0.3,0) -- (0,1);
\draw (-0.3,0) -- (0,-1);
\end{scope}
\begin{scope}[shift={(0,4)}]
\node[dotnode] at (0,0) {};
\draw (-0.3,1) -- (0.3,1);
\draw[overline] (0,-1.2) -- (0,1.2);
\draw[overline] (-1,0) -- (1,0);
\draw[overline] (-0.3,-1) -- (0.3,-1);
\draw[thin_overline={1.5}] (-0.1,-0.1)
	.. controls +(-135:0.3cm) and +(-135:0.4cm) .. (-0.15,0.15)
	.. controls +(45:0.4cm) and +(45:0.3cm) .. (0.1,0.1);
\draw[->] (-1.5,-1) -- (-3.5,-2.5);
\node at (-3,-0.8) {top};
\node at (-3,-1.2) {torihedron};
\draw[->] (1.5,-1) -- (3.5,-2.5);
\node at (3,-0.8) {bottom};
\node at (3,-1.2) {torihedron};
\node at (2,1) {left};
\node at (2,0.6) {augmentation};
\end{scope}
\begin{scope}[shift={(0,-4)}]
\node[dotnode] at (0,0) {};
\draw (-0.3,1) -- (0.3,1);
\draw[overline] (0,-1.2) -- (0,1.2);
\draw[overline] (-1,0) -- (1,0);
\draw[overline] (-0.3,-1) -- (0.3,-1);
\draw[thin_overline={1.5}] (-0.1,0.1)
	.. controls +(135:0.3cm) and +(135:0.4cm) .. (0.15,0.15)
	.. controls +(-45:0.4cm) and +(-45:0.3cm) .. (0.1,-0.1);
\draw[->] (-1.5,1) -- (-3.5,2.5);
\node at (-3,1.4) {bottom};
\node at (-3,1) {torihedron};
\draw[->] (1.5,1) -- (3.5,2.5);
\node at (3,1.4) {top};
\node at (3,1) {torihedron};
\node at (1.5,-1.2) {right};
\node at (1.5,-1.6) {augmentation};
\end{scope}
\end{tikzpicture}
\caption{Bow-tie modifications and its relation to torihedra}
\label{f:bow-tie-modification}
\end{figure}

\begin{figure}[ht]
\begin{tikzpicture}
\begin{scope}[shift={(0,0)}]
\node[dotnode] at (0,0) {};
\draw (-1,0) -- (0,0);
\draw (-1,-0.4) -- (0,0);
\draw (0,-1.2) -- (0,1);
\draw (0,1) -- (0.2,1.2);
\draw (0,1) -- (-0.2,1.2);
\draw (-0.3,-1) -- (0.3,-1);
\node at (-0.15,0.15) {$v$};
\node at (0.15,0.5) {$e$};
\node at (0.2,-0.5) {$e'$};
\end{scope}
\begin{scope}[shift={(2.5,0)}]
\draw[->] (-1,0) -- (1,0);
\node at (0,0.7) {left bow-tie};
\node at (0,0.3) {modification};
\end{scope}
\begin{scope}[shift={(5,0)}]
\node[dotnode] at (0,0) {};
\draw (-1,0) -- (0.3,0);
\draw (-1,-0.4) -- (-0.3,0);
\draw (0,-1.2) -- (0,1);
\draw (0,1) -- (0.2,1.2);
\draw (0,1) -- (-0.2,1.2);
\draw (-0.3,-1) -- (0.3,-1);
\draw (-0.3,0) -- (0,1);
\draw (0.3,0) -- (0,-1);
\end{scope}
\end{tikzpicture}
\caption{Another example of (left) bow-tie modification}
\label{f:bow-tie-modification-general}
\end{figure}

\label{d:diagram-bowtie}
\end{definition}

\begin{define}
We say a twist region is \emph{right-augmented} if,
when both strands are (locally) oriented
so that they cross the augmentation disk in the same direction,
the crossing is a right-handed half-twist.
We say a twist is \emph{left-augmented} if it is not right-augmented.
(See \figref{f:bow-tie-modification}).
\end{define}

In other words, we have the tautological-sounding fact that
augmenting a right-handed twist region
is a right-handed augmentation.

We can recover $L$ from the link diagram of $K$
together with labels at twist regions indicating left- or right-augmentation.

\begin{definition}
\label{d:tori-decomp-graph}
Let $L$ be a link obtained from augmenting an alternating link $K$
with a cellular link diagram $D = D(K)$.
We define the \emph{top/bottom bow-tie graph of $L$} as follows.
(See \figref{f:tori-decomp-graph} for illustrations.)

Let $D'$ be the graph obtained from $D$
by collapsing each augmented twist region of $K$ to a vertex.
Clearly, $D'$ is the link diagram of a link $K'$
obtained from $K$ by removing half-twists from each augmented twist region
until one crossing remains.
Let $v_t$ denote a vertex of $D'$ corresponding to
an augmented twist region of $K$,
and let $v_c$ denote a vertex of $D'$ corresponding to
a crossing of $K$ not in an augmented twist region.

Orient the edges of $D'$ to point from an undercrossing
to an overcrossing.
Label the two outgoing edges at vertex $v$
(which corresponds to a crossing or a twist region)
by $e_v^{(1)}, e_v^{(2)}$ (in arbitrary order).
For each left- (resp. right-) augmented twist region $t$,
we perform a left (resp. right) bow-tie modification to
$v_t, e_{v_t}^{(1)}, e_{v_t}^{(2)}$.

We call the resulting graph the \emph{top bow-tie graph of $L$},
denoted by $\Gamma_T(L)$.
If we had oriented the edges of $D'$ the other way,
and subsequently performed the same operations,
we obtain another graph,
which we call the \emph{bottom bow-tie graph of $L$},
denoted by $\Gamma_B(L)$.

Note that the non-bow-tie faces of $\Gamma_T(L)$ and $\Gamma_B(L)$
are naturally identified with the faces of $D'$
(as the bow-tie modification procedure does not remove faces).
\end{definition}

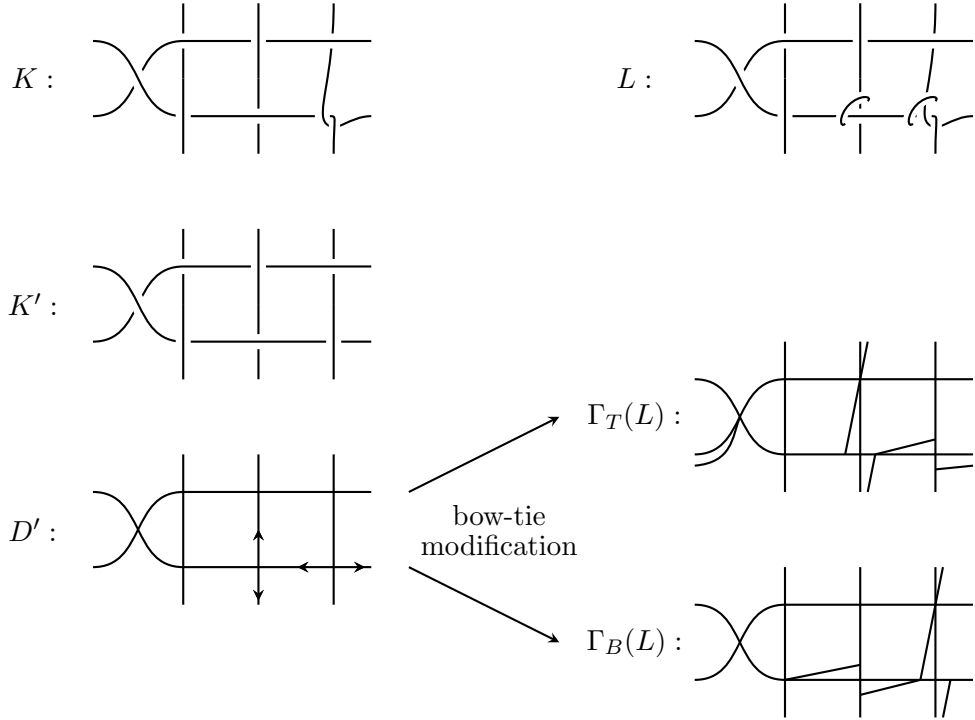
\begin{figure}[ht]
\begin{tikzpicture}
\begin{scope}[shift={(0,0)}]
\node at (-2,0.5) {$K:$};
\draw (0,-0.5) -- (0,1.5);
\draw (1,-0.5) -- (1,1.5);
\draw (2,1.5)
	.. controls +(-90:1cm) and +(135:0.2cm) .. (1.9,-0.1)
	.. controls +(-45:0.2cm) and +(180:0.3cm) .. (2.5,0);
\draw[overline] (-1.2,0)
	.. controls +(0:0.7cm) and +(180:0.7cm) .. (0,1)
	-- (2.5,1);
\draw[overline] (-1.2,1)
	.. controls +(0:0.7cm) and +(180:0.7cm) .. (0,0)
	-- (1.75,0);
\draw[thin_overline={1.5}] (1.95,0)
	.. controls +(0:0.1cm) and +(90:0.5cm) .. (2,-0.5);
\draw[overline] (0,-0.5) -- (0,0.5);
\draw[overline] (1,0.5) -- (1,1.5);
\end{scope}
\begin{scope}[shift={(8,0)}]
\node at (-2,0.5) {$L:$};
\draw (0,-0.5) -- (0,1.5);
\draw (1,-0.5) -- (1,1.5);
\draw (2,1.5)
	.. controls +(-90:1cm) and +(135:0.2cm) .. (1.9,-0.1)
	.. controls +(-45:0.2cm) and +(180:0.3cm) .. (2.5,0);
\draw[overline] (-1.2,0)
	.. controls +(0:0.7cm) and +(180:0.7cm) .. (0,1)
	-- (2.5,1);
\draw[overline] (-1.2,1)
	.. controls +(0:0.7cm) and +(180:0.7cm) .. (0,0)
	-- (1.75,0);
\draw[thin_overline={1.5}] (1.95,0)
	.. controls +(0:0.1cm) and +(90:0.5cm) .. (2,-0.5);
\draw[overline] (0,-0.5) -- (0,0.5);
\draw[overline] (1,0.5) -- (1,1.5);
\draw[thin_overline={1.5}] (0.8,-0.1)
	.. controls +(-135:0.1cm) and +(-135:0.2cm) .. (0.85,0.15)
	.. controls +(45:0.2cm) and +(45:0.1cm) .. (1.1,0.2);
\draw[thin_overline={1.5}] (1.7,-0.1)
	.. controls +(-135:0.1cm) and +(-135:0.2cm) .. (1.75,0.15)
	.. controls +(45:0.2cm) and +(45:0.1cm) .. (2,0.2);
\end{scope}
\begin{scope}[shift={(0,-3)}]
\node at (-2,0.5) {$K':$};
\draw (0,-0.5) -- (0,1.5);
\draw (1,-0.5) -- (1,1.5);
\draw (2,-0.5) -- (2,1.5);
\draw[overline] (-1.2,0)
	.. controls +(0:0.7cm) and +(180:0.7cm) .. (0,1)
	-- (2.5,1);
\draw[overline] (-1.2,1)
	.. controls +(0:0.7cm) and +(180:0.7cm) .. (0,0)
	-- (2.5,0);
\draw[overline] (0,-0.5) -- (0,0.5);
\draw[overline] (1,0.5) -- (1,1.5);
\draw[overline] (2,-0.5) -- (2,0.5);
\end{scope}
\begin{scope}[shift={(0,-6)}]
\node at (-2,0.5) {$D':$};
\draw (0,-0.5) -- (0,1.5);
\draw[midarrow_rev={0.1},midarrow={0.5}]  (1,-0.5) -- (1,1.5);
\draw (2,-0.5) -- (2,1.5);
\draw (-1.2,0)
	.. controls +(0:0.7cm) and +(180:0.7cm) .. (0,1)
	-- (2.5,1);
\draw[midarrow_rev={0.8},midarrow={0.98}] (-1.2,1)
	.. controls +(0:0.7cm) and +(180:0.7cm) .. (0,0)
	-- (2.5,0);
\draw[->] (3,1) -- (5,2);
\draw[->] (3,0) -- (5,-1);
\node at (4.2,0.7) {bow-tie};
\node at (4.2,0.3) {modification};
\end{scope}
\begin{scope}[shift={(8,-4.5)}]
\node at (-2,0.5) {$\Gamma_T(L):$};
\draw (0,-0.5) -- (0,1.5);
\draw (1,-0.5) -- (1,1.5);
\draw (2,-0.5) -- (2,1.5);
\draw (-1.2,0)
	.. controls +(0:0.7cm) and +(180:0.7cm) .. (0,1)
	-- (2.5,1);
\draw (-1.2,1)
	.. controls +(0:0.7cm) and +(180:0.7cm) .. (0,0)
	-- (2.5,0);
\draw (0.8,0) -- (1,1);
\draw (1.2,0) -- (1.1,-0.5);
\draw (1.1,1.5) -- (1,1);
\draw (1.2,0) -- (2,0.2);
\draw (2,-0.2) -- (2.5,-0.15);
\draw (-1.2,-0.15)
	.. controls +(5:0.5cm) and +(-110:0.3cm) .. (-0.6,0.5);
\end{scope}
\begin{scope}[shift={(8,-7.5)}]
\node at (-2,0.5) {$\Gamma_B(L):$};
\draw (0,-0.5) -- (0,1.5);
\draw (1,-0.5) -- (1,1.5);
\draw (2,-0.5) -- (2,1.5);
\draw (-1.2,0)
	.. controls +(0:0.7cm) and +(180:0.7cm) .. (0,1)
	-- (2.5,1);
\draw (-1.2,1)
	.. controls +(0:0.7cm) and +(180:0.7cm) .. (0,0)
	-- (2.5,0);
\draw (0,0) -- (1,0.2);
\draw (1,-0.2) -- (1.8,0);
\draw (1.8,0) -- (2,1);
\draw (2.2,0) -- (2.1,-0.5);
\draw (2.1,1.5) -- (2,1);
\end{scope}
\end{tikzpicture}
\caption{Constructing the top and bottom bow-tie graphs,
using bow-tie modifications}
\label{f:tori-decomp-graph}
\end{figure}

\begin{prop}\label{p:tori_decomp}
Let $K$ be an alternating link in the thickened torus
with a cellular link diagram,
and let $L$ be an augmented link obtained from $K$.
There is a decomposition of the complement,
$(\torus \times I) - L$, into two ideal torihedra.

Moreover, the graphs of the torihedra are
the top and bottom bow-tie graphs from
\defref{d:tori-decomp-graph},
$\Gamma_T(L)$ and $\Gamma_B(L)$,
respectively.

We call this the
\emph{torihedral decomposition of
the link complement of $L$ in the thickened torus}.
\end{prop}

\begin{proof}
As mentioned before, we will be combining
Menasco's method using crossing edges at each crossing
and Lackenby's ``cut-slice-flatten'' method
on augmentation sites.

Let $L = K \cup C$, with $C$ being the collection of crossing circles.
Arrange $L$ in the following way:
place the circle components in $C$ perpendicular to
the projection plane $\torus \times \{0\}$,
and leave the remaining part of the link $K \subseteq L$
lying in the projection plane (except at crossings of $K$).
Thus, the projection of $L$ onto the projection plane
will be a diagram $D(K)$ of $K$
together with line segments corresponding to crossing circles.

We now place a \emph{crossing edge} at each crossing of $K$,
connecting the top and bottom strands at the crossing
(see \figref{f:crossing-edges}).
We also three horizontal edges for each crossing circle
(see \figref{f:cut-slice-flatten}, leftmost diagram).

\begin{figure} 
\begin{tikzpicture}
\begin{scope}[shift={(0,0)}]
\draw (2,-1)
	.. controls +(30:0.7cm) and +(-150:0.3cm) .. (3,-0.7)
	.. controls +(30:0.3cm) and +(-150:0.7cm) .. (4,0);
\draw[overline] (2,0)
	.. controls +(-30:0.7cm) and +(150:0.3cm) .. (3,-0.3)
	.. controls +(-30:0.3cm) and +(150:0.7cm) .. (4,-1);
\draw (0,1)
	.. controls +(-30:0.7cm) and +(150:0.3cm) .. (1,0.3)
	.. controls +(-30:0.3cm) and +(150:0.7cm) .. (2,0);
\draw[overline] (0,0)
	.. controls +(30:0.7cm) and +(-150:0.3cm) .. (1,0.7)
	.. controls +(30:0.3cm) and +(-150:0.7cm) .. (2,1);
\draw[color=red,midarrow={0.7}] (1,0.7) -- (1,0.3);
\draw[color=red,midarrow={0.7}] (3,-0.3) -- (3,-0.7);
\node[inner sep=0] (x) at (0.7,-0.7) {\tiny crossing};
\node[inner sep=0] (y) at (0.7,-1.0) {\tiny edges};
\draw[line width=0.4pt] (x) -- (1,0.5);
\draw[line width=0.4pt] (x) -- (3,-0.6);
\draw[->] (-0.3,0) -- (-1.3,0);
\node at (-0.8,0.25) {\small bottom};
\draw[->] (4.3,0) -- (5.3,0);
\node at (4.8,0.2) {\small top};
\end{scope}
\begin{scope}[shift={(5.5,0)}]
\node[emptynode] (a) at (2.6,-0.75) {};
\node[emptynode] (b) at (3.4,-0.4) {};
\node[emptynode] (c) at (1.4,0.25) {};
\node[emptynode] (d) at (0.6,0.6) {};
\draw (2,-1)
	.. controls +(30:0.7cm) and +(180:0.1cm) .. (a);
\draw (b)
	.. controls +(50:0.1cm) and +(-150:0.7cm) .. (4,0);
\draw (2,0)
	.. controls +(-30:0.7cm) and +(150:0.3cm) .. (3,-0.4)
	.. controls +(-30:0.3cm) and +(150:0.7cm) .. (4,-1);
\draw (0,1)
	.. controls +(-30:0.7cm) and +(130:0.1cm) .. (d);
\draw (c)
	.. controls +(0:0.1cm) and +(150:0.7cm) .. (2,0);
\draw (0,0)
	.. controls +(30:0.7cm) and +(-150:0.3cm) .. (1,0.6)
	.. controls +(30:0.3cm) and +(-150:0.7cm) .. (2,1);
\draw[color=red,midarrow={0.7}] (3,-0.4) -- (a);
\draw[color=red,midarrow={0.7}] (3,-0.4) -- (b);
\draw[color=red,midarrow={0.7}] (1,0.6) -- (c);
\draw[color=red,midarrow={0.7}] (1,0.6) -- (d);
\end{scope}
\begin{scope}[shift={(-1.5,0)},rotate=180]
\node[emptynode] (a) at (2.6,-0.75) {};
\node[emptynode] (b) at (3.4,-0.4) {};
\node[emptynode] (c) at (1.4,0.25) {};
\node[emptynode] (d) at (0.6,0.6) {};
\draw (2,-1)
	.. controls +(30:0.7cm) and +(180:0.1cm) .. (a);
\draw (b)
	.. controls +(50:0.1cm) and +(-150:0.7cm) .. (4,0);
\draw (2,0)
	.. controls +(-30:0.7cm) and +(150:0.3cm) .. (3,-0.4)
	.. controls +(-30:0.3cm) and +(150:0.7cm) .. (4,-1);
\draw (0,1)
	.. controls +(-30:0.7cm) and +(130:0.1cm) .. (d);
\draw (c)
	.. controls +(0:0.1cm) and +(150:0.7cm) .. (2,0);
\draw (0,0)
	.. controls +(30:0.7cm) and +(-150:0.3cm) .. (1,0.6)
	.. controls +(30:0.3cm) and +(-150:0.7cm) .. (2,1);
\draw[color=red,midarrow_rev={0.7}] (3,-0.4) -- (a);
\draw[color=red,midarrow_rev={0.7}] (3,-0.4) -- (b);
\draw[color=red,midarrow_rev={0.7}] (1,0.6) -- (c);
\draw[color=red,midarrow_rev={0.7}] (1,0.6) -- (d);
\end{scope}
\end{tikzpicture}
\caption{Splitting and flattening crossing edges.}
\label{f:crossing-edges}
\end{figure}
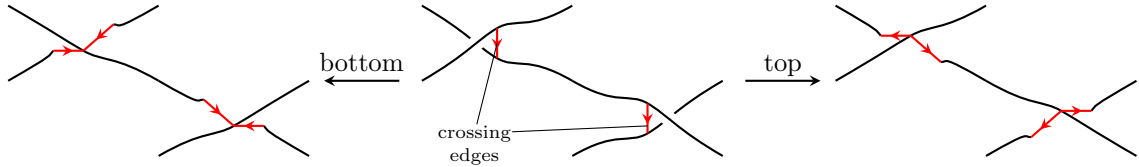

View the link from the point at infinity from the top end
($\torus \times \{1\}$) of the thickened torus.
At each crossing of $K$,
push the top strand towards the bottom strand,
splitting the crossing edge into two
identical edges and spreading them apart as in
\figref{f:crossing-edges}.
Note that after this operation, the link (which is no longer a link)
for both top and bottom look the same,
but a crossing edge associated to a crossing
is pushed in different directions;
this contributes to a ``$2\pi/n$ twist''
when gluing back the faces of the top and bottom torihedra.

Now for each crossing circle $c$, consider a spanning
(twice-punctured) disk $B_c$.
The following operations are depicted in \figref{f:cut-slice-flatten}.
$B_c$ intersects the projection plane $\torus \times \{0\}$,
cutting $B_c$ into two pieces $B_c^+,B_c^-$,
each being an ideal triangle.
We then slice along the disk $B_c$, turning it into two copies,
$B_c^{(1)},B_c^{(2)}$ (in no particular order);
each copy $B_c^{(j)}$ is also cut horizontally into two pieces,
$B_c^{(j),+},B_c^{(j),-}$.
We untwist all crossings in the twist region
which $c$ encircles,
rotating $B_c^{(j)}$ by $180^\circ$ for each crossing.
Then we flatten the disks;
the crossing circle is shrunk to a point,
as it is at infinity.

\begin{figure}[ht]
\input{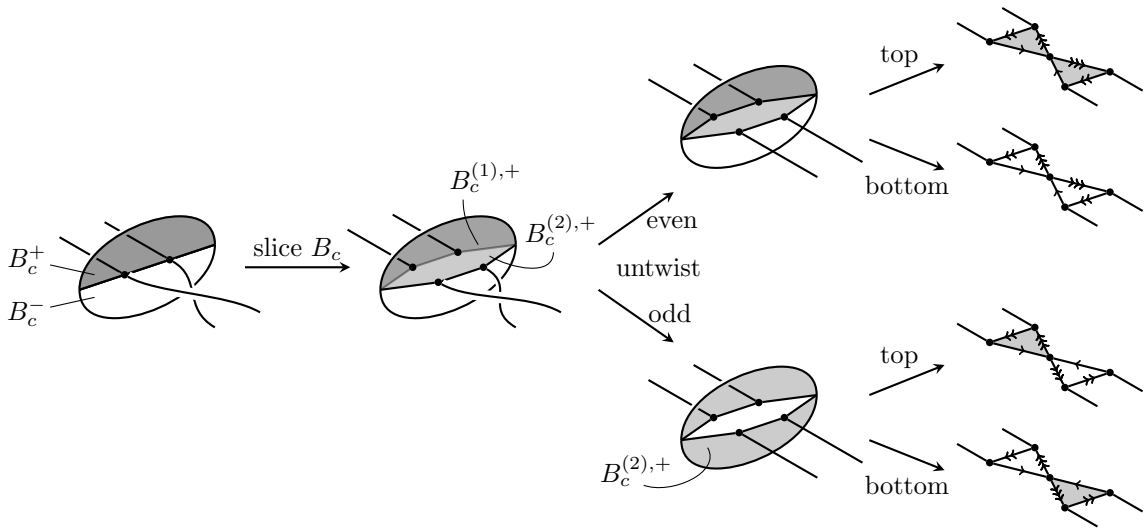}
\caption{Cut-slice-flatten;
the even and odd refer to the number of crossings
in that twist region
(we only draw one crossing in the first diagram);
the top and bottom graphs are glued back,
gray face to gray face, white face to white face;
note that the graph does not depend on the
parity of the number of crossings, but the gluing is different}
\label{f:cut-slice-flatten}
\end{figure}

Finally, we shrink all remaining segments of the link $L$
to ideal vertices.
It is easy to see that the top and bottom graphs
are exactly $\Gamma_T(L)$ and $\Gamma_B(L)$ from
\defref{d:tori-decomp-graph}.
To recover the link complement,
we glue bow-tie to bow-tie as described in \figref{f:cut-slice-flatten},
and glue each non-bow-tie face to its natural counterpart
(see \defref{d:tori-decomp-graph}),
with a ``$2\pi/n$'' twist as discussed before.

\end{proof}

The Figures \ref{fig:step_one} to \ref{fig:top-bottom}
depict an example which decomposes the link (C) of
\figref{fig:Augmentations}.

\begin{figure}[h] 
\centering
\includegraphics[height=3cm]{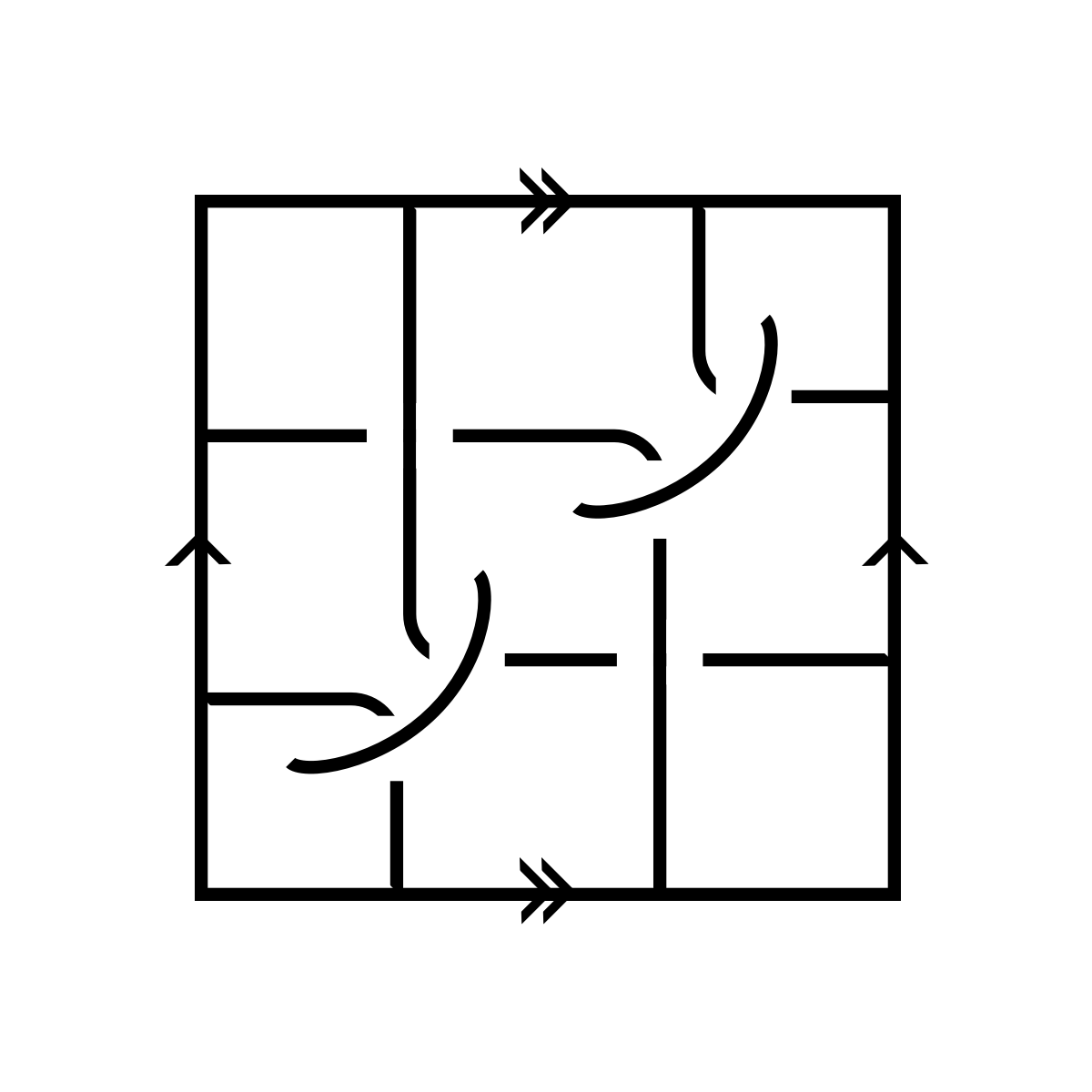}
\includegraphics[height=3cm]{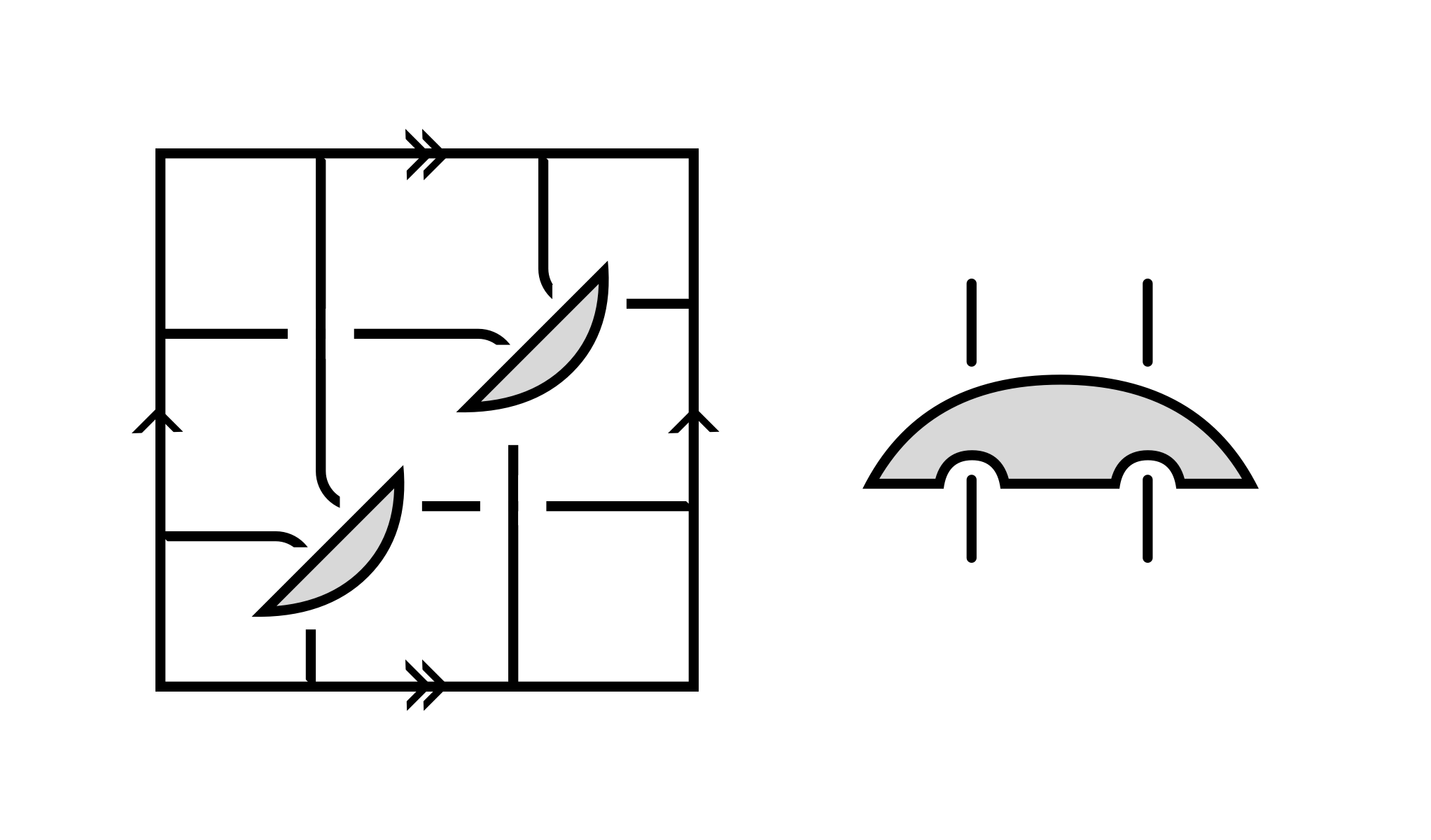}
	\caption{Each crossing circle bounds a twice-punctured disk;
	the rightmost figure shows half of the disk.}
	\label{fig:step_one}
\end{figure}
 
\begin{figure}[h] 
\centering 
\includegraphics[height=3cm]{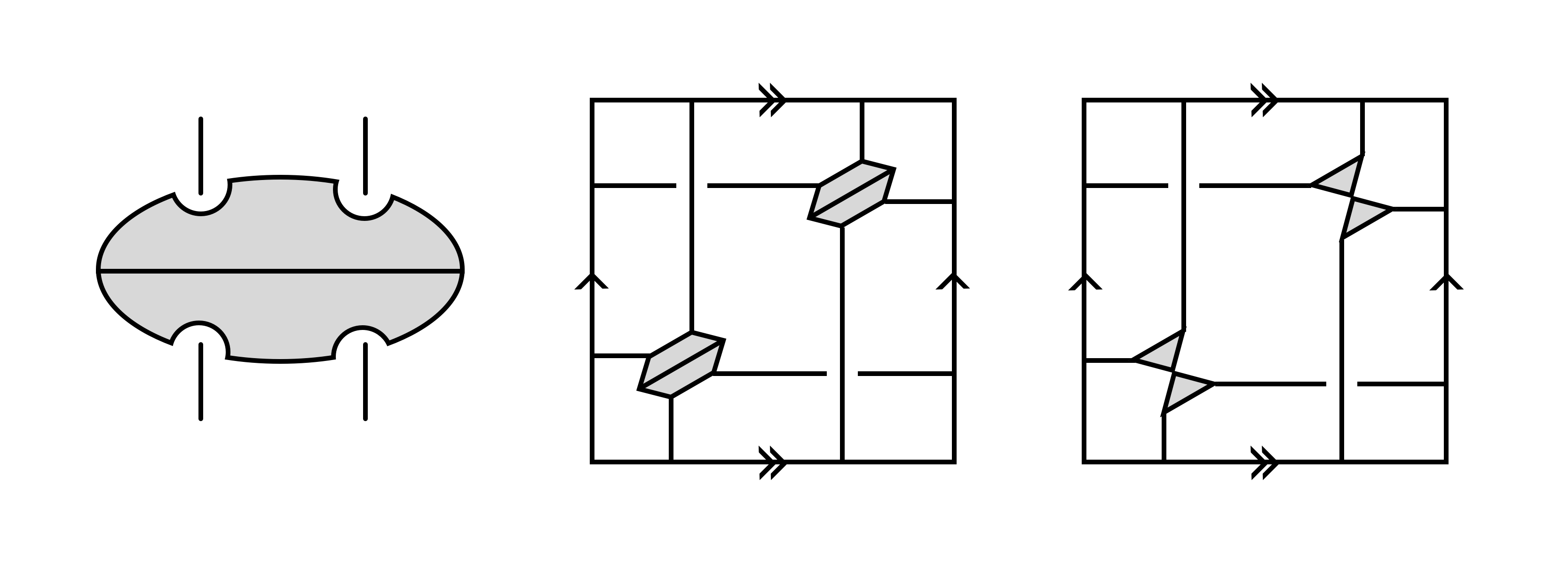} 
	\caption{We split the disk and collapse the arc of each
 crossing circle to ideal vertices;
 the leftmost figure shows the view from the top
 of a half disk being spread apart.}
	\label{fig:step_two}
\end{figure}

\begin{figure}[h] 
\centering 
\includegraphics[height=3cm]{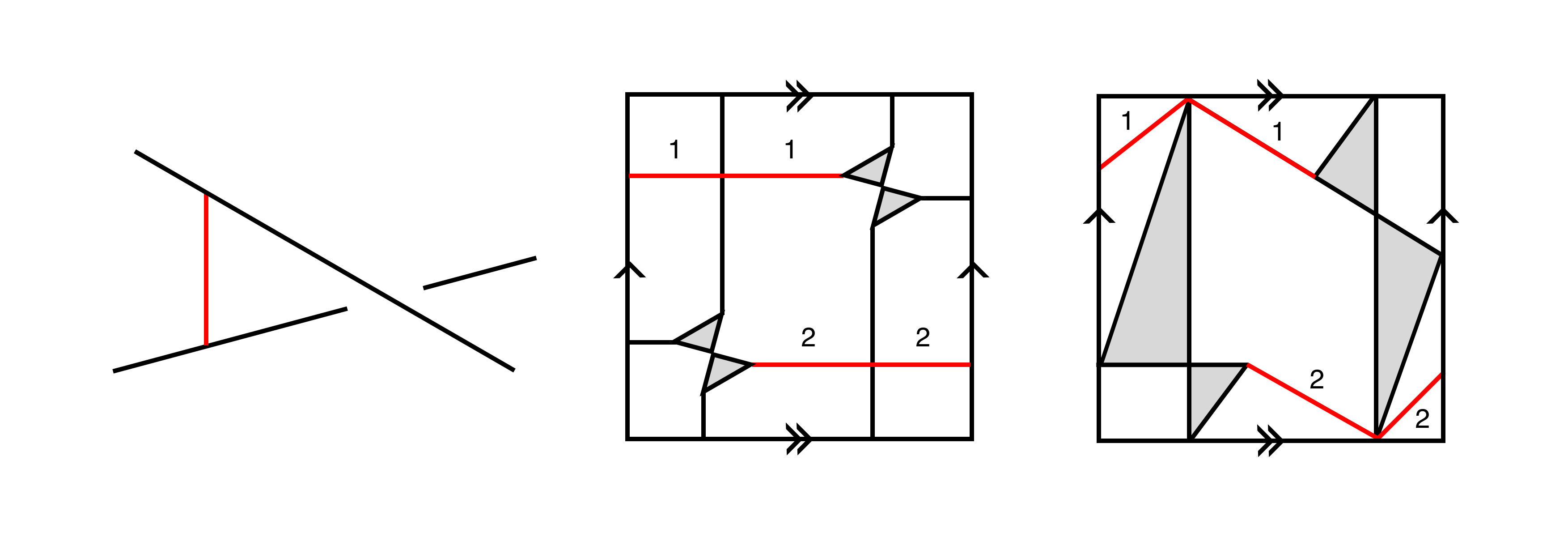} 
\caption{Left: The crossing arc is the edge in red.
Middle: Picture of splitting the crossing edge
and bow-ties.
Right: The link components are pushed off to infinity
(diagrammatically, they are shrunk to points).}
	\label{fig:step_three}
\end{figure}

\begin{figure}[h] 
\centering 
\includegraphics[height=3cm]{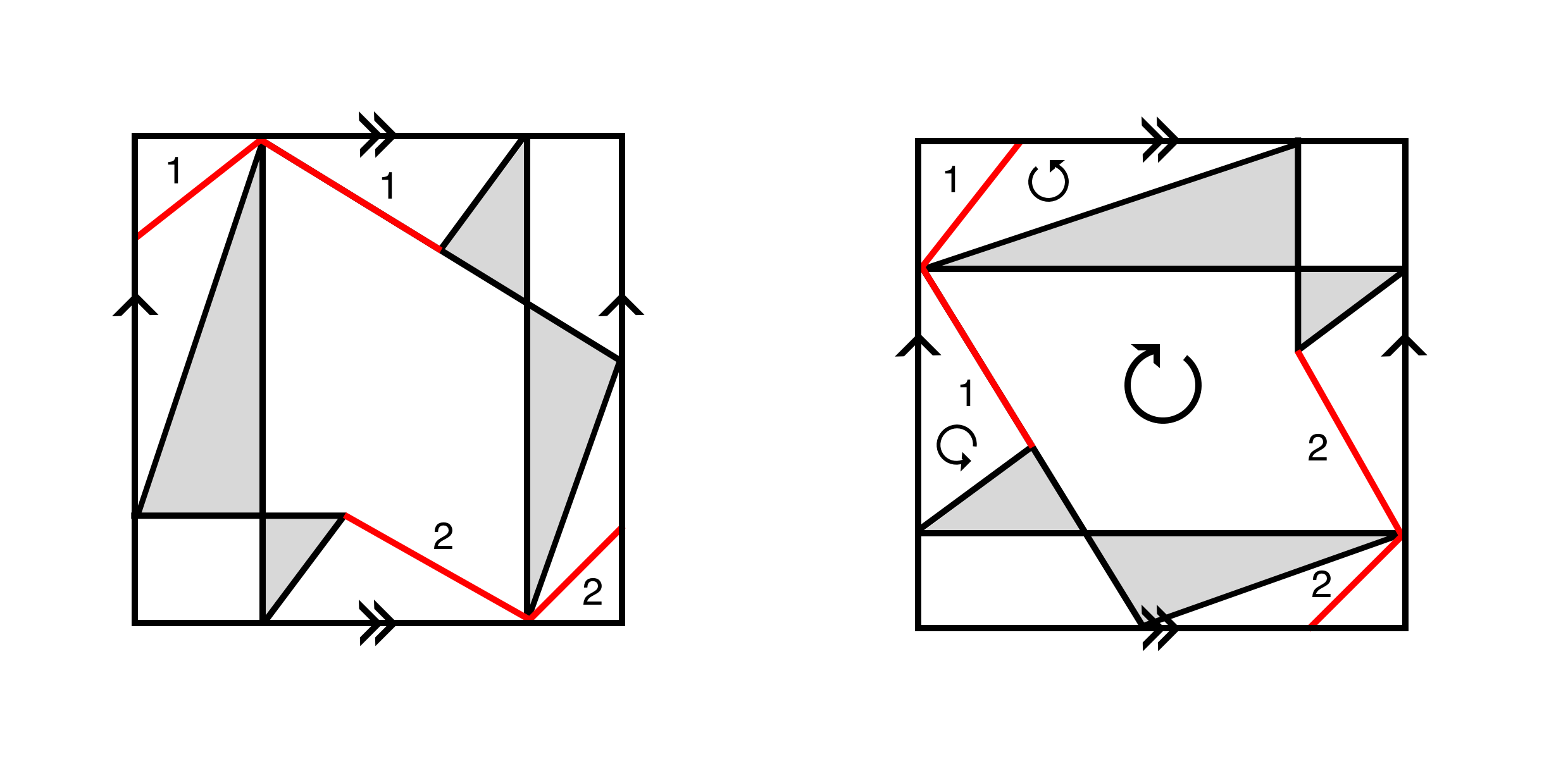} 
	\caption{Left: The top torihedron.
	Right: The bottom torihdron with rotation indicating face gluing}
\label{fig:top-bottom}
\end{figure}

\begin{remark}
We note that our main \thmref{t:auglink_hyp}
requires that all non-trivial twist regions be augmented,
but \prpref{p:tori_decomp} does not require it.
\label{r:unnecessary-augment-bigon}
\end{remark}

\begin{definition}
An \emph{angled torihedron} $(\sT, \theta_\bullet^*)$
is a torihedron $\sT$ with
an assignment of an \emph{interior dihedral angle}
$\theta_e^* \in [0,\pi]$ to each edge $e$ of $G(\sT)$
such that for each vertex $v \in G(\sT)$,
$\sum_{e \ni v} \theta_e^* = (\deg(v) - 2)\pi$.
We also denote $\theta_e = \pi - \theta_e^*$,
so $\sum_{e \ni v} \theta_e = 2\pi$;
we refer to $\theta_e$ as the \emph{exterior dihedral angle}.
For brevity, we write dihedral angle to mean 
interior dihedral angle.

We say $(\sT, \theta_\bullet^*)$ is \emph{degenerate}
if $\theta_e^* = 0$ for some edge;
we say it is \emph{non-degenerate} otherwise.
\end{definition}

One may ask for the pyramidal decomposition of a torihedron
to ``respect" angles. The following definitions,
in particular an ``angle splitting'', make sense of this.

\begin{samepage}
\begin{define}
An \emph{angled ideal tetrahedron} is an ideal tetrahedron
with an assignment of an
interior dihedral angle $\theta_e^*$ to each edge $e$, such that
\begin{itemize}
\item each dihedral angle is in $[0, \pi]$;
\item for each tetrahedron, opposite edges have equal dihedral angles;
\item the three distinct interior angles at edges incident to one vertex sum to $\pi$.
\end{itemize}

We say an angled ideal tetrahedron is \emph{degenerate} if
one dihedral angle is 0; we say it is \emph{non-degenerate} otherwise.
\end{define}
\end{samepage}

\begin{define}
A \emph{base-angled ideal pyramid}
is a pyramid whose base is an $n$-gon, $n \geq 3$,
and each boundary edge $e_i$ of the base face is assigned a dihedral angle
$\alpha_i \geq 0$ such that their sum is $\sum \alpha_i = \pi$.
The vertical edge $e_i'$ that meets $e_i$ and $e_{i+1}$
is automatically assigned the dihedral angle $\pi - \alpha_i - \alpha_{i+1}$.

We say a base-angled ideal pyramid is \emph{degenerate} if
$\alpha_i = 0$ for some $i$; we say it is \emph{non-degenerate} otherwise.
\end{define}

Clearly, the dihedral angles of an ideal hyperbolic pyramid
make it a base-angled ideal pyramid
(with $\alpha_i = \vphi_{e_i}$);
it is not hard to see that the converse is true:
simply consider a circumscribed polygon such that the side $e_i$
subtends an angle of $2\alpha_i$ at the center,
and take the ideal hyperbolic pyramid over it in upper-half space.
Also, an angled ideal tetrahedron is simply a base-angled ideal pyramid
with base a triangle, and with no preferred face.

\begin{definition}
An \emph{angle-splitting} of an angled torihedron $(\sT,\theta_\bullet^*)$
is an assignment of an angle $\vphi_{\vec{e}}$ to each
oriented edge $\vec{e}$, such that

\begin{itemize}
\item for each edge $e$,
$\theta_e^* = \vphi_{\vec{e}} + \vphi_{\cev{e}}$,
where $\cev{e}$ is the opposite orientation on $e$,
\item for each face $f$,
$\sum_{\vec{e} \in \del f} \vphi_{\vec{e}} = \pi$,
where $\vec{e} \in \del f$ is the edge in the boundary of $f$
taken with outward-orientation
(see Convention \ref{cvn:ornt-edge}).
\end{itemize}

Equivalently, an angle-splitting is a decomposition of
$\sT$ into base-angled pyramids,
one for each face $F$ of $G(\sT)$, such that
the interior dihedral of the edge $\vec{e} \in \del F$
is $\vphi_{\vec{e}}$.

We also say that $\vphi_\bullet$ is an angle-splitting
of the edge-labeled graph $(G(\sT), \theta_\bullet^*)$.

We say that an angle-splitting is \emph{degenerate}
if $\vphi_{\vec{e}} = 0$ for some oriented edge $\vec{e}$;
it is \emph{non-degenerate} otherwise.
\end{definition}

\begin{convention}
The outward-orientation on the boundary of a face
is the orientation such that the face is to the left of the boundary.
An assignment/label on an oriented edge $\vec{e}$
(for example, $\vphi_{\vec{e}}$)
will usually be drawn to the left of that edge.
\label{cvn:ornt-edge}
\end{convention}

\begin{remark}
These $\theta$'s are the same as the $\theta$'s in
\cite{BandS},
and angle-splittings $\vphi_\bullet$'s
are the same as their ``coherent angle system''.
\end{remark}

\begin{lemma}
\label{l:pyramid_decomp}
Let $P_n$ be a base-angled ideal pyramid, and suppose we are given a
decomposition of the base face into triangles by adding new edges.  One gets an
obvious corresponding triangulation of $P_n$, where a new face is added for each
new edge. Then there is an assignment of a dihedral angle to each edge of each
ideal tetrahedron in this triangulation such that
\begin{itemize}
\item each tetrahedron is an angled ideal tetrahedron;
\item the sum of dihedral angles around each new edge is $\pi$;
\item the dihedral angles of the edges of the original base face are the same as
	before.
\end{itemize} 
Moreover, if $P_n$ is non-degenerate,
then the resulting angled tetrahedra are also non-degenerate.
\end{lemma}

\begin{proof}
Induct on $n$; there is nothing to prove for the base case $n=3$.

The proof is essentially given in Figure \ref{f:ideal_pyramid_arg}.
We spell it out here in words.

\begin{figure}
\includegraphics[height=5cm]{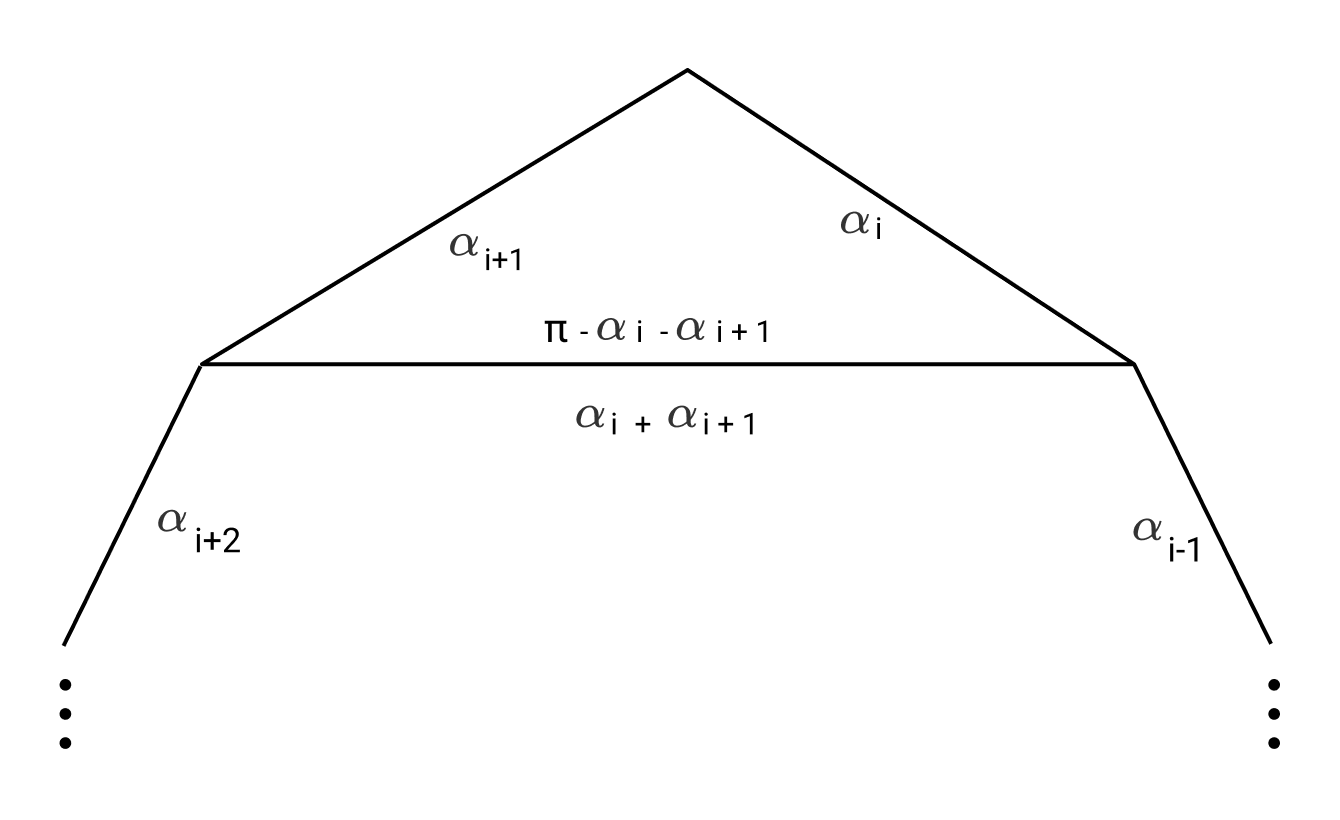}
\caption{Angle-splitting on a polygonal face of the graph}
\label{f:ideal_pyramid_arg}
\end{figure}

Suppose the edges are labeled $e_i$, for an edge
which goes between vertices $v_i$ and $v_{i+1}$,
and suppose $e_i$ is assigned dihedral angle $\alpha_i$.
Let $e'$ be a new edge added to the base face of $P_n$
such that it separates the base face into a triangle and
an $(n-1)$-gon;
suppose the sides of the triangle are
$e_i, e_{i+1}$, and $e'$.
The new face corresponding to $e'$ separates $P_n$ into
an ideal tetrahedron $T$ and an ideal pyramid $P_{n-1}$.
We assign the dihedral angle of $\pi - \alpha_i - \alpha_{i+1}$
to $e'$ in $T$, and assign $\alpha_i + \alpha_{i+1}$ to $e'$ in $P_{n-1}$.
Clearly the sum of dihedral angles condition is satisfied
in $T$ and $P_{n-1}$.
It remains to check that the dihedral angles assigned to the vertical (non-base)
edges are correct.
For the vertical edge associated to $v_j$ for $j \neq i, i+2$,
there is nothing to check;
for $j = i$, the dihedral angles are
$\pi - \alpha_i - (\pi - \alpha_i - \alpha_{i-1})$
in $T$ and $\pi - \alpha_{i-1} - (\alpha_i + \alpha_{i+1})$ in $P_{n-1}$,
which sum to $\pi - \alpha_i - \alpha_{i+1}$;
it is similar for $j = i+2$.

Non-degeneracy of the resulting angled tetrahedra
follows easily from the observation that the angles
assigned to each side of a new edge is simply the sum
of the angles of original edges on the other side.
\end{proof}

\section{Hyperbolicity of Augmented Links}
\label{s:hyperbolicity}
Thurston introduced a method for finding the
unique complete hyperbolic metric for a given 3-manifold $M$
with boundary consisting of tori \cite{Thurston}. 
Thurston wrote down a system of gluing and consistency equations
which can be translated to equations involving
angles for a triangulation of $M$ whose solutions correspond to the
complete hyperbolic metric on the interior of $M$.
Casson and Rivin separated Thurston's
gluing equations into a linear and non-linear part \cite{Casson-Rivin}.
Angle structures are solutions to the linear part of
Thurston's gluing equations.
By \thmref{t:hyperbolic-angle-str}
(\cite[Theorem 1.1]{FG-angles}),
to prove hyperbolicity of a link complement,
it suffices to find an angle structure
on a triangulation of the link complement.

\begin{define}
Let $M$ be an orientable 3-manifold with boundary consisting of tori.
An angle structure on an ideal triangulation $\tau$ of $M$
is an assignment of a dihedral angle to each edge
of each tetrahedron, such that
\begin{itemize}
\item each tetrahedron is a non-degenerate angled ideal tetrahedron,
\item around each edge of $\tau$, the dihedral angles sum to $2\pi$.
\end{itemize}
\end{define}


\begin{theorem}\cite[Theorem 1.1]{FG-angles}
\label{t:hyperbolic-angle-str}
Let $M$ be a 3-manifold with a triangulation that admits an angle structure.
Then $M$ is hyperbolic.
\end{theorem}

For a hyperbolic link $K$ in $\torus \times I$, we show
that the link $L$ obtained from augmenting $K$ is hyperbolic.
The idea is to start with a graph from the torihedral decomposition
of the link $K$ which will give us a graph on each torihedron with an angle
assignment of $\pi/2$ to each edge \cite{CKP2}.
By \prpref{p:tori_decomp},
there is a torihedral decomposition of the complement of the augmented link $L$.
Using those angles from $K$,
we then assign new angles locally to edges of torihedra from a torihedral 
decomposition of $L$
and decompose them into base-angled pyramids which can be decomposed 
into tetrahedra, thus obtaining an angle structure on a triangulation.

We need the following theorem,
adapted from \cite[Theorem 4]{BandS},
specialized to genus 1 surfaces:

\begin{theorem}{\cite[Theorem 4]{BandS}}
Let $\Gamma = (V,E)$ be a graph on the torus,
and let $\check{\Gamma} = (F,\check{E})$ be the dual graph,
with $\check{E}$ being naturally identified with $E$.
Let $f \in (0,\pi)^E$
be a function on the set of edges $E$
that sums to $2\pi$ around each vertex of $V$;
let $f^*(e) = \pi - f(e)$.

There exists a non-degenerate
angle-splitting of $(\Gamma,f^*)$
if and only if the following is satisfied:

\begin{quotation}
Suppose we cut the torus along a subset of edges in the dual graph
$\check{\Gamma}$, obtaining one or more pieces;
Then for any piece that is a disk,
the sum of $f$ over the edges in the boundary
of the piece is at least $2\pi$,
with equality if and only if the piece
contains exactly one vertex of $\Gamma$.
\end{quotation}

\label{t:bs-thm4}
\end{theorem}

The original theorem \cite[Theorem 4]{BandS}
proves that a circle pattern combinatorially equivalent to $\Gamma$
exists; a circle pattern naturally yields
an angle-splitting (which they call a coherent angle system).

\begin{theorem}
\label{t:auglink_hyp}
Let $K$ be a weakly prime, alternating link
in the thickened torus
whose diagram is cellular and has no bigons.
Let $L$ be a link obtained from augmenting $K$.
Then $L$ is hyperbolic.

More generally, if $K$ is as above with a twist-reduced diagram
containing bigons,
and $L$ is obtained by augmenting $K$
such that for every twist region with at least one bigon
is augmented,
then $L$ is hyperbolic.
\end{theorem}

\begin{proof}
By \prpref{p:tori_decomp}, $\toruscomp{L}$ can be
obtained by gluing two torihedra $\sT_T(L),\sT_B(L)$
with graphs $\Gamma_T(L),\Gamma_B(L)$.

Recall that $\Gamma_T(L),\Gamma_B(L)$ are
obtained by bow-tie modifications of the diagram $D'$
of a link $K'$
(see \defref{d:tori-decomp-graph}).
Assign to each edge $e$ of $D'$ the angle $\theta_e = \pi/2$
(so that $\theta_e^* = \pi/2$ too).
This assignment has the property that
the sum of dihedral angles around each identified edge
(in the complement of $K'$) is $2\pi$.
This can be achieved because of the hypotheses on bigons
in our theorem statement.

Using the fact that $K'$ is weakly prime
(which easily follows from $K$ being weakly prime),
it is not hard to see that the condition on cocycles
of \thmref{t:bs-thm4} is satisfied by this assignment.
Thus, there exists a non-degenerate angle-splitting
$\vphi_\bullet$ of $(D',\theta_\bullet^*)$.

Now we perform the bow-tie modifications to obtain
$\Gamma_T(L),\Gamma_B(L)$.
For each step (i.e. each bow-tie modification),
we show how to modify the $\theta^*$ assignments
and how to get angle-splittings.
Say we perform such a modification
at some vertex $v$ and two edges $e^{(1)},e^{(2)}$.
We assign new $\theta_\bullet^*$ angles to
the resulting bow-tie modification graph
as in \figref{f:bowtie_angles}.
Note that the sum of $\theta$ (not $\theta^*$)
around each vertex is still $2\pi$.
\figref{f:bowtie_angles2} shows an angle-splitting
of this assignment.

\begin{figure}
\includegraphics[width=7cm]{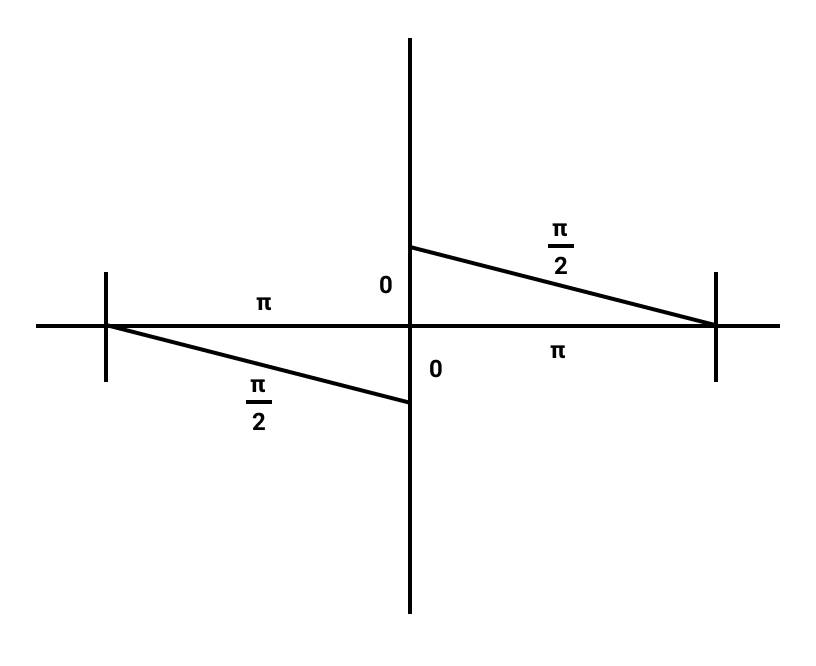}
\caption{Assignments of $\theta^*$ to edges of a bow-tie
	corresponding to a left augmentation site;
	the long edges are assigned $\pi$,
	the short edges are assigned 0,
	and the diagonal edges are assigned $\pi/2$
	(same for a right augmentation).
	}
\label{f:bowtie_angles}
\end{figure}

\begin{figure}
\begin{tabular}{cc}
\includegraphics[width = 5cm]{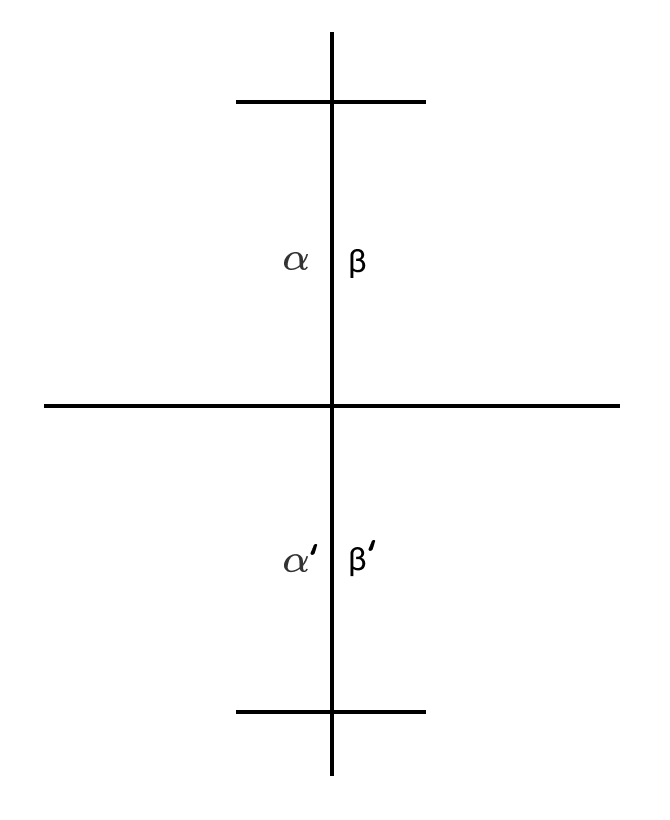}&
\includegraphics[width = 5cm]{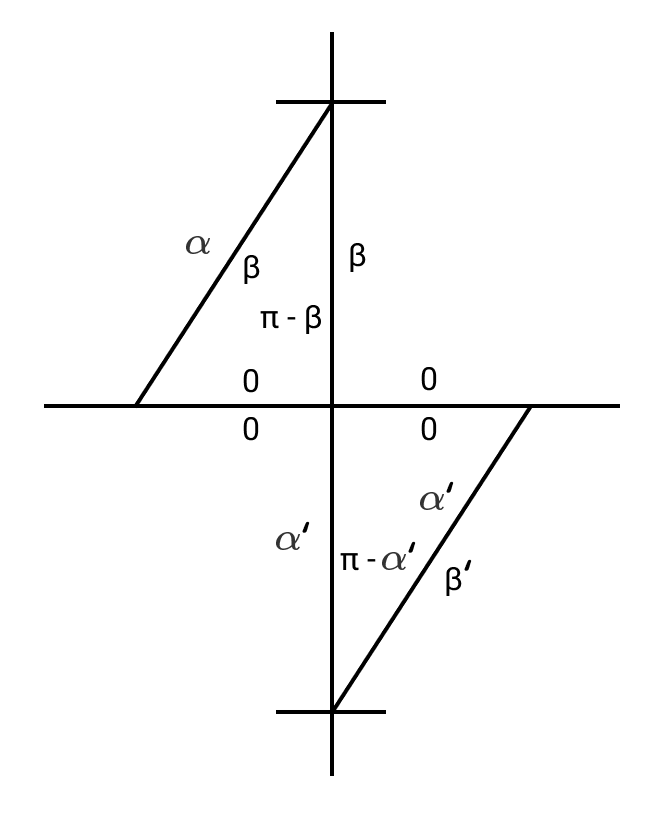}\\
(a)&(b)
\end{tabular}
\caption{(a) Angle splitting before augmentation 
(b) Angle splitting for bowtie corresponding to
left bow-tie modification/augmentation.
For right bow-tie modification/augmentation,
just flip both diagrams (a) and (b) horizontally.}
\label{f:bowtie_angles2}
\end{figure}

We check that upon gluing the top and bottom torihedra,
the sum of interior dihedral angles $\theta^*$
around each edge is $2\pi$:
crossing edges have $\theta^* = \pi/2$,
and appear four times, twice in each torihedron,
while for bow-tie edges,
simply check for half-twist and non-half-twist cases
separately.

Now we have a decomposition of the two torihedra into
degenerate base-angled pyramids
(recall that that means some of the
interior dihedral angles $\theta^*$ are 0);
since we need the pyramids to be non-degenerate,
we modify the graph on the torihedra
and the angle assignments to make all $\theta^*$ non-zero
as follows.

We first modify the graphs on the torihedra by adding edges
to them for some extra ``flexibility''.
Consider a face $f$ of $\Gamma_T(L)$ that is not from a bow-tie.
Suppose the corresponding face $\bar{f}$ of $D'$
had vertices $v_1,\ldots,v_n$ in counter-clockwise order.
Note that $f$ may meet a vertex twice,
but we label each occurrence with its own index.
We label the edges of $f$ by $e_{i,0}$, $e_{i,\pi}$, or $e_i$,
depending on whether the $\theta^*$ of that edge
is $0$, $\pi$, or $\pi/2$ respectively.
More precisely, for a vertex $v_i$ corresponding to a crossing
of $K$ that is not augmented,
we label $e_{v_i}^{(.)}$ by $e_i$
(here $e_{v_i}^{(.)}$ is $e_{v_i}^{(1)}$ or $e_{v_i}^{(2)}$,
whichever meets $\bar{f}$; see \defref{d:tori-decomp-graph}).
For a vertex $v_i$ that corresponds to a twist region of $K$,
if the crossing circle ``faces'' $f$,
then $f$ meets a diagonal edge of the bow-tie corresponding to $v_i$,
and we label it $e_i$
(for example if $f$ is the top left face
in \figref{f:bowtie_angles2} (a));
if not, then $f$ meets a short and long edge
of the corresponding bow-tie,
and we label them by $e_{i,0}$ and $e_{i,\pi}$,
respectively
(for example if $f$ is the top right face
in \figref{f:bowtie_angles2} (b)).

If $\bar{f}$ does not have vertices of the latter kind,
i.e. if $f$ does not meet short or long edges,
then we will not modify $f$.
So assume that $\bar{f}$ does have such a vertex,
and suppose it is right-augmented
(the other case is treated similarly).
Then for all right-augmented vertices $v_i$ of $\bar{f}$,
$f$ would meet the short, long edges $e_{i,0},e_{i,\pi}$,
while for all left-augmented vertices,
$f$ would meet the diagonal edges.
In particular,
the edges $e_{i,0}, e_{i,\pi}$ always appear in counter-clockwise order.

Suppose, after cyclically reindexing, $v_1,\ldots,v_k$
is a maximally contiguous subsequence of right-augmented vertices
of $D'$ around $\bar{f}$;
the edges around $f$ would start
$e_{1,0} \cm e_{1,\pi} \cm e_{2,0} \cm e_{2,\pi} \cm \ldots
	\cm e_{k,0} \cm e_{k,\pi} \cm \ldots$.
We add new edges across $f$ as follows
(see \figref{f:adding_edges};
ignore the + and - signs for now):

\begin{itemize}
\item \textbf{Case $k=n$:} (i.e. every vertex of $\bar{f}$ is right-augmented.)
In this case, we do nothing.

\item \textbf{Case: There is only one such maximal contiguous subsequence:}

\begin{itemize}
\item \textbf{Subcase: $k=1$:}
We add an edge that goes across $e_{1,0},e_{1,\pi},e_2$
(in the sense that the new edge separates the edges of $f$ into two sets,
one of them being those three edges;
since $n\geq 3$, this edge is new).

\item \textbf{Subcase: $k \geq 2$:}
We add an edge across $e_{1,0},e_{1,\pi}$
and another edge across $e_{2,0},e_{2,\pi},e_{3,0},\ldots,e_{k,\pi}$
(these two edges do not form a bigon because we've ruled out $k=n$).
\end{itemize}

\item \textbf{Case: There are multiple such maximal contiguous subsequences.}
We just add edges as in the previous case for each contiguous subsequence,
except one special case: when the edges of $f$ are exactly
$e_{1,0},e_{1,\pi},e_2,e_{3,0},e_{3,\pi},e_4$,
we add only one edge separating the first three edges from the other three.
(This prevents formation of a bigon.)
\end{itemize}

\begin{figure}
\begin{tikzpicture}
\begin{scope}[shift={(0,0)}]
\node[dotnode] (a) at (-1,2) {};
\node[dotnode] (b) at (-1,0.9) {};
\node[dotnode] (c) at (-0.1,0) {};
\node[dotnode] (d) at (1,0) {};
\draw (a) -- (b) -- (c) -- (d);
\draw[dotted] (a) -- +(60:1cm);
\draw (a) -- +(60:0.5cm);
\draw[dotted] (d) -- +(30:1cm);
\draw (d) -- +(30:0.5cm);
\draw[densely dotted] (a) to[out=-10,in=100] (d);
\node[inner sep=0, outer sep=0] (e10) at (-1.6,1.0) {\tiny $e_{1,0}$};
\draw[line width=0.4pt] (e10) -- (-1,1.1);
\node at (-1.15,1.45) {\tiny $+$};
\node at (-0.85,1.45) {\tiny $+$};
\node[inner sep=0, outer sep=0] (e1pi) at (-1.3,0.3) {\tiny $e_{1,\pi}$};
\draw[line width=0.4pt] (e1pi) -- (-0.8,0.7);
\node at (-0.65,0.35) {\tiny $-$};
\node at (-0.45,0.55) {\tiny $-$};
\node at (0.2,-0.15) {\tiny $e_2$};
\node at (0.4,0.15) {\tiny $+2$};
\node at (0.7,1.6) {\tiny new edge};
\node at (0.25,1.1) {\tiny $-2$};
\end{scope}
\begin{scope}[shift={(5,0)}]
\node[dotnode] (e) at (-1.2,1.6) {};
\node[dotnode] (a) at (-1,0.6) {};
\node[dotnode] (b) at (-0.1,0) {};
\node[dotnode] (c) at (1,0) {};
\node[dotnode] (d) at (2.3,1.9) {};
\draw (e) -- (a) -- (b) -- (c);
\draw[dotted] (e) -- +(80:0.9cm);
\draw (e) -- +(80:0.5cm);
\draw[dotted] (c) -- +(30:0.9cm);
\draw (c) -- +(30:0.5cm);
\draw[dotted] (d) -- +(-100:0.9cm);
\draw (d) -- +(-100:0.5cm);
\draw[dotted] (d) -- +(100:0.9cm);
\draw (d) -- +(100:0.5cm);
\draw[densely dotted] (e) to[out=-10,in=90] (b);
\draw[densely dotted] (b) to[out=70,in=180] (d);
\node[inner sep=0, outer sep=0] (e10) at (-1.8,0.6) {\tiny $e_{1,0}$};
\draw[line width=0.4pt] (e10) -- (-1.04,0.8);
\node at (-1.25,1.0) {\tiny $+$};
\node at (-0.95,1.1) {\tiny $+$};
\node[inner sep=0, outer sep=0] (e1pi) at (-1.6,0.2) {\tiny $e_{1,\pi}$};
\draw[line width=0.4pt] (e1pi) -- (-0.82,0.48);
\node at (-0.5,0.4) {\tiny $-$};
\node at (-0.6,0.2) {\tiny $-$};
\node[inner sep=0, outer sep=0] (e20) at (1,-0.6) {\tiny $e_{2,0}$};
\draw[line width=0.4pt] (e20) -- (0.8,0);
\node at (0.5,-0.15) {\tiny $+$};
\node at (0.5,0.15) {\tiny $+$};
\node[inner sep=0, outer sep=0] (e2pi) at (1.5,-0.3) {\tiny $e_{2,\pi}$};
\draw[line width=0.4pt] (e2pi) -- (1.2,0.1);
\node at (1.3,0.3) {\tiny $-$};
\node at (1.5,0.1) {\tiny $-$};
\node[inner sep=0, outer sep=0] (ekpi) at (2.9,1.7) {\tiny $e_{k,\pi}$};
\draw[line width=0.4pt] (ekpi) -- (2.25,1.7);
\node at (2.4,1.5) {\tiny $-$};
\node at (2.05,1.55) {\tiny $-$};
\node at (-0.3,1.3) {\tiny $-2$};
\node at (0.8,1.7) {\tiny $+2$};
\node[inner sep=0, outer sep=0] (n) at (0.4,2.5) {\tiny new edges};
\draw[line width=0.4pt] (n) -- (-0.7,1.4);
\draw[line width=0.4pt] (n) -- (1.5,1.8);
\end{scope}
\end{tikzpicture}
\caption{Adding edges to non-bow-tie faces;
the $+,-$ labels are short for $+1,-1$ respectively.}
\label{f:adding_edges}
\end{figure}
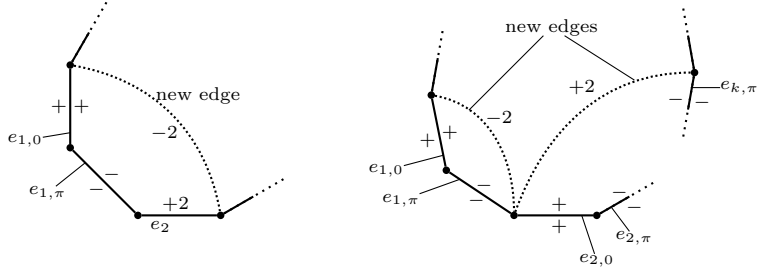

This way we obtain a new graph $\Gamma_T'$, which defines a
new torihedron $\sT_T'$.
We make $\sT_T'$ angled using the angles from $\sT_T(L)$ for old edges,
and putting $\theta^* = \pi$ for all new edges.

We can get an angle-splitting for $(\Gamma_T',\theta^*)$,
using the angle-splitting for $\Gamma_T(L)$ for old edges,
and for a new edge $e$ (as in \figref{f:adding_edges})
that cuts through some face $f$ of $\Gamma_T(L)$,
we assign $\vphi_{\vec{e}} = \sum \vphi_{\vec{e'}}$,
where the sum is over all the edges $\vec{e'} \in \del f$
on the other side of $\vec{e}$.
By non-degeneracy of the angle-splitting on $D'$,
these assignments to the new edges are all non-zero.

Now fix some small $\veps > 0$.
Let $\vphi_{\vec{e}}' = \vphi_{\vec{e}} + x \cdot \veps$,
where $x$ is the label
(in \figref{f:adding_edges}) on $\vec{e}$
(set $x=0$ if unlabeled).
Let $\theta_e'^* = \vphi_{\vec{e}}' + \vphi_{\cev{e}}'$.
It is easy to check that the sum of labels
around each vertex is 0,
hence $\theta'^*$
defines an angled torihedron $(\sT_T',\theta'^*)$.
For each face of $\Gamma_T'$, the $+/-$ labels
on the inner side of boundary edges cancel out:
for bow-tie faces, the short edge gets a $+1$ and the long edge
gets $-1$, while for non-bow-tie faces,
it is clear from \figref{f:adding_edges}.
Hence, $\vphi'$ is an angle-splitting of $(\sT_T',\theta'^*)$.
Furthermore, all shorts edges
(which are the only edges with $\vphi_{\vec{e}} = 0$)
have a $+1$ label on each side,
so $\vphi'$ is non-degenerate.

Now we perform the same operations for the bottom torihedron,
adding new edges to $\Gamma_B(L)$ to get $\Gamma_B'$
in the same manner;
note that left and right augmentations are switched,
so that the order of $e_{i,0},e_{i,\pi}$ are switched.
Thus all the $+/-$ labels in \figref{f:adding_edges}
should have switched signs.
We also get a nondegenerate angle-splitting $\vphi'$
of an angled torihedron $(\sT_B',\theta'^*)$.

By construction, under the gluing of
$\sT_T(L)$ to $\sT_B(L)$,
the new edges added to $\Gamma_T(L)$
are glued to the new edges added to $\Gamma_B(L)$,
since they are added by the same procedure.
As noted before, upon gluing $\sT_T(L)$ to $\sT_B(L)$
the sum of exterior dihedral angles $\theta^*$ around each edge
is $2\pi$.
This clearly remains true after adding the new edges
(they're labeled $\pi$ on each torihedron).
Again by construction,
the $+/-$ labels coming from the top and bottom diagrams get canceled out.
Thus, upon gluing $\sT_T'$ to $\sT_B'$,
the sum of new exterior dihedral angles $\theta'^*$
around each edge is still $2\pi$.

Finally, we obtain a triangulation with an angle structure as follows.
For each face of $\Gamma_T'$ that has more than three sides,
we arbitrarily decompose it into triangles
and apply \lemref{l:pyramid_decomp}
to obtain a triangulation of $\sT_T'$ into non-degenerate angled tetrahedra;
perform the corresponding decomposition for faces of
$\Gamma_B'$ and obtain a triangulation of $\sT_B'$
into non-degenerate angled tetrahedra.
These triangulations glue up into a triangulation of $\toruscomp{L}$
with an angle structure.
Thus, by \thmref{t:hyperbolic-angle-str},
$L$ is hyperbolic.
\end{proof}

\begin{remark}
By applying \thmref{t:auglink_hyp} judiciously,
one may prove hyperbolicity of augmented links
that do not satisfy its hypotheses as stated.


Say we have a link $K$ that satisfies
the stricter hypotheses of \thmref{t:auglink_hyp}
(no bigons).
Consider a crossing circle $C$ around two parallel strands
that do not meet at a crossing.
The addition of $C$ to $K$ is not an augmentation
as in \defref{d:augmentation},
and so \thmref{t:auglink_hyp} does not directly apply.
However, we may consider the related link $K''$
where those two strands have a full twist,
so that the addition of $C$ to $K''$ is an augmentation.
Now \thmref{t:auglink_hyp} applies to $K'' \cup C$,
and thus $K \cup C$ is also hyperbolic.

Note that augmenting in the sense of \defref{d:augmentation}
will not turn a non-weakly prime link into a weakly prime link,
but the above extended notion of augmenting might.

Another way to squeeze more out of \thmref{t:auglink_hyp}
is to consider augmenting twist regions
in the ``transverse direction''.
That is, suppose $K$ has only one twist region $T$,
which is left-handed,
but we perform a right-handed augmentation
at each vertex of $T$
(instead of a single left-handed augmentation, which is
required for \thmref{t:auglink_hyp}).
Consider the link $K''$ where each vertex of $T$
is replaced by a right-handed twist region of length 2
(two bigons).
Then the twist region $T$ ``is'' no longer a twist region
in $K''$, and we may freely apply \thmref{t:auglink_hyp})
and conclude that the ``transversely augmented'' $K$
is hyperbolic.
\end{remark}

\bibliographystyle{plain}
\bibliography{references-ak}


\end{document}